\documentclass[prodmode,acmtoms]{acmsmall} 
\def\sizefig{0.4}
\usepackage{listings}

\usepackage[utf8]{inputenc}
\usepackage[T1]{fontenc}
\usepackage{lmodern}
\usepackage{graphicx}  
\usepackage{amsmath} 

\newenvironment{psmallmatrix}
  {\left(\begin{smallmatrix}}
  {\end{smallmatrix}\right)}
\usepackage{amssymb}  
\usepackage{color}
\usepackage{enumitem}
\usepackage{myalgo}
\usepackage{hyperref}
\usepackage{enumerate}
\usepackage{multirow}

\newcommand{\new}[1]{\textcolor{black}{#1}}

\newcommand{\R}{\mathbb{R}}
\newcommand{\F}{\mathbb{F}}

\newcommand{\Mc}{\mathcal{M}}

\newcommand{\T}{\text{T}}
\newcommand{\N}{\mathbb{N}}
\newcommand{\x}{\mathbf{x}}
\newcommand{\e}{\mathbf{e}}
\newcommand{\y}{\mathbf{y}}
\newcommand{\z}{\mathbf{z}}

\def\L{\mathbf{L}}

\newcommand{\M}{\mathbf{M}}

\def\Rb{\mathbf{R}}
\def\B{\mathbf{B}}
\def\E{\mathbf{E}}
\def\I{\mathbf{I}}
\def\K{\mathbf{K}}
\def\bone{\mathbf{1}}
\def\bzero{\mathbf{0}}

\def\X{\mathbf{X}}

\newcommand{\A}{\mathbf{A}}

\renewcommand{\prec}{\text{prec}}

\DeclareMathOperator{\diag}{diag}

\newcommand{\iaboundfun}[2]{\mathtt{ia\_bound}(#1, #2)}
\newcommand{\iabound}{\mathtt{ia\_bound}}
\newcommand{\lowerboundfun}[3]{\mathtt{sdp\_bound}(#1, #2, #3)}
\newcommand{\lowerbound}{\mathtt{sdp\_bound}}

\newcommand{\sthreefp}{\mathtt{s3fp}}

\newcommand{\fpbern}{\mathtt{FPBern}}
\newcommand{\nlopt}{\mathtt{NLopt}}

\newcommand{\mvbeta}{\mathtt{mvbeta}}
\newcommand{\geneig}{\mathtt{geneig}}
\newcommand{\robustsdp}{\mathtt{robsdp}}
\newcommand{\realtofloat}{\mathtt{Real2Float}}
\newcommand{\fpsdp}{\mathtt{FPSDP}}
\newcommand{\matlab}{\textsc{Matlab}}
\newcommand{\yalmip}{\textsc{Yalmip}}
\newcommand{\mosek}{\textsc{Mosek}}

\newcommand{\bop}{\mathtt{bop}}
\newcommand{\coq}{\text{\sc Coq}}

\newcommand{\rosa}{\mathtt{Rosa}}

\newcommand{\fptaylor}{\mathtt{FPTaylor}}

\if{
\makeatletter
\newcommand*{\circled}{\@ifstar\circledstar\circlednostar}
\newcommand*{\squared}{\@ifstar\squaredstar\squarednostar}
\makeatother

\newcommand*\circledstar[1]{%
  \tikz[baseline=(C.base)]
    \node[%
      fill,
      circle,
      minimum size=1.em,
      text=white,
      inner sep=0.5pt
    ](C) {\texttt{#1}};%
}
\newcommand*\circlednostar[1]{%
  \tikz[baseline=(C.base)]
    \node[%
      draw,
      circle,
      minimum size=1.em,
      inner sep=0.5pt
    ](C) {\texttt{#1}};%
}
\newcommand*\squaredstar[1]{%
  \tikz[baseline=(C.base)]
    \node[%
      fill,
      rectangle,
      minimum size=1.em,
      text=white,
      inner sep=0.5pt
    ](C) {\texttt{#1}};%
}
\newcommand*\squarednostar[1]{%
  \tikz[baseline=(C.base)]
    \node[%
      draw,
      rectangle,
      minimum size=1.em,
      inner sep=0.5pt
    ](C) {\texttt{#1}};%
}
}\fi

\newtheorem{property}[theorem]{Property}

\if{
\newtheorem{theorem}{Theorem}[section]
\newtheorem{lemma}[theorem]{Lemma}

\newtheorem{property}[theorem]{Property}

\theoremstyle{plain}

\newtheorem{example}{Example}

}\fi

\begin{document}


\title{Interval Enclosures of Upper Bounds of Roundoff
Errors using Semidefinite Programming}

\author{VICTOR MAGRON
\affil{CNRS Verimag}
}

\if{
\author[verimag]{Victor Magron\corref{cor}}
\ead{victor.magron@imag.fr}

\cortext[cor]{Corresponding author}
\address[verimag]{CNRS Verimag; 700 av Centrale 38401 Saint-Martin d'Hères, France}
}\fi


\begin{abstract}
A longstanding problem related to floating-point implementation of numerical programs is to provide efficient yet precise analysis of output errors.

We present a framework to compute lower bounds on \new{largest} absolute roundoff errors, for a particular rounding model. This method applies to numerical programs implementing polynomial functions with box constrained input variables. 
\new{Our study is based on three different hierarchies, relying respectively on generalized eigenvalue problems, elementary computations and  semidefinite programming (SDP) relaxations.} This is complementary of over-approximation frameworks, consisting of obtaining upper bounds \new{on the largest absolute} roundoff error. Combining the results of both frameworks allows to get  enclosures for upper bounds on roundoff errors.

\new{The under-approximation framework provided by the third hierarchy is based on a new sequence of convergent robust SDP approximations for certain classes of polynomial optimization problems. Each problem in this hierarchy can be solved exactly} via SDP. By using this hierarchy, one can provide a monotone non-decreasing sequence of lower bounds converging  to the absolute roundoff error of a program implementing a polynomial function, applying for a particular rounding model.

We investigate the efficiency and precision of our method on non-trivial polynomial programs coming from space control, optimization and computational biology.
\end{abstract}




\if{
}\fi


\acmformat{Victor Magron, 2017. 
Interval Enclosures of Upper Bounds of Roundoff
Errors using Semidefinite Programming.}

\begin{bottomstuff}
This work has been partially supported by the LabEx PERSYVAL-Lab (ANR-11-LABX-0025-01) funded by the French program ``Investissement d'avenir'' and by the European Research Council (ERC) ``STATOR'' Grant Agreement nr. 306595.

Author's addresses: V. Magron, CNRS Verimag, 700 av Centrale 38401 Saint-Martin d'Hères FRANCE.
\end{bottomstuff}

\maketitle
\section{INTRODUCTION}
\label{sec:intro}
Over the last four decades, numerical programs have extensively been written and executed with finite precision implementations~\cite{Dekker71}, often relying on  single or double floating-point numbers to perform fast computation. A ubiquitous related issue, especially in the context of critical system modeling, is to precisely analyze the absolute gap between the real and floating-point output of such programs. 
The existence of a possibly high roundoff error gap is a consequence of multiple rounding occurrences, happening most likely while performing operations with finite precision systems, such as IEEE 754 standard arithmetic~\cite{IEEE}.

The present study focuses on computing a lower bound on the  \new{largest}  absolute roundoff error for a particular rounding model, while executing a program implementing a multivariate polynomial function $f$, \new{with a priorly fixed bracketing}. For these programs, each input variable takes a value within a given closed interval. 
\new{We consider a simple (multiplicative) rounding model of the variables and elementary operations involved in $f$. Executing the program in floating-point precision  leads to the computation of a rounded expression $\hat{f}(\x,\e)$ depending on the input variables $\X$ and additional roundoff error variables $\e := (e_1,\dots,e_m)$.  With unit roundoff $\varepsilon$, the value of each $e_j$ can take values in $[-\varepsilon,\varepsilon]$.
Stated formally, our goal is to compute a lower bound on the largest absolute value of the expression $\hat{f}(\x,\e) - f(\x)$ for all possible values of $\x$ in $\X$ and each possible value of $e_j$ in $[-\varepsilon,\varepsilon]$.
}
Exact resolution of this problem is nontrivial as it requires to compute the maximum of a polynomial, which is \new{known} to be NP-hard~\cite{laurent2009sums} in general.


%
Several existing methods allow to obtain {\em lower bounds} of roundoff errors. The easiest way to obtain such a bound on the maximum of a given function is to evaluate this function at several points within the function input domain before taking the minimum over all evaluations.
Testing approaches aim at finding the inputs causing the worst error. 
\new{Such techniques often rely on guided random testing as in $\sthreefp$~\cite{Chiang14s3fp}, or heuristic search as in {\sc Precimonious}~\cite{Precimonious}, {\sc CORAL}~\cite{Borges12Test}. 
}

\new{The $\sthreefp$ tool implements the so-called {\em Binary Guided Random Testing} (BGRT) method. BGRT relies on shadow value executions and configuration evaluations. A shadow value execution is the execution of a program under certain precision settings to compute either absolute or relative roundoff errors.  A configuration is a mapping from program inputs to corresponding range of values.  An initial configuration is splitted recursively into tighter configurations, where tightness is determined thanks to shadow value executions. The BGRT algorithm starts with a configuration, enumerates a sub-part of its tighter configurations, in order to pick the set of inputs causing  the (locally) maximial high floating-point errors.
For more details, we refer the interested reader to~\cite[Section~3]{Chiang14s3fp}.
}

\new{
The {\sc Precimonious} tool aims at assisting users to execute numerical programs with priorly prescribed accuracy in a more efficient way. For this, the tool performs automated tuning of the floating-point precision related to the elementary operations involved in the program. In the best scenario, {\sc Precimonious} outputs a program with an optimial configuration, that is the setting which uses the least bit precisions resulting in the best performance improvement over all complying configurations. This goal is pursued by performing local search over a subset of the input variables provided by the user. {\sc Precimonious} relies on the delta-debugging algorithm for local search. This consists of dividing the sets of possible configuration changes and increasing the number of subsets inductively when no improvement occurs. }
\new{
While {\sc Precimonious} does not take into account the correlation between variables, {\sc CORAL} relies on meta-heuristic solvers based on genetic algorithms~\cite{Goldberg89} and particle-swarm optimization~\cite{Kennedy95} to handle complex mathematical constraints. 
}

Lower bound computed with testing are complementary with tools providing validated {\em upper bounds}. These tools are mainly based on interval arithmetic (e.g. {\sc Gappa}~\cite{Daumas10}, {\sc Fluctuat}~\cite{fluctuat}, $\rosa$~\cite{Darulova14Popl}) or  methods coming from global optimization such as Taylor approximation in $\fptaylor$ by~\cite{fptaylor15}, Bernstein expansion in $\fpbern$ by~\cite{tacas16}. The recent framework by~\cite{toms16}, related to the $\realtofloat$ software package, employs semidefinite programming (SDP) to obtain a hierarchy of upper bounds converging to the absolute roundoff error. 
This hierarchy is derived from the general moment-sum-of-squares hierarchy (also called Lasserre's hierarchy) initially provided by~\cite{Lasserre01moments} in the context of polynomial optimization. 
\new{At each step of Lasserre's hierarchy, one can either rely on moments or sum-of-squares (SOS) to compute a certified upper bound on the maximum (or similary a lower bound on the minimum) for a given objective polynomial function $f$ under a set of polynomial inequality constraints $\K$. In the unconstrained case, the underlying idea  is that if one can decompose $f$ into a sum of squares (SOS) then it is straighforward to prove that this polynomial is nonnegative. 
In the constrained case, the idea is to write $f$ as a weighted SOS decomposition, where the weights are the polynomials involved in set of constraints $\K$. This also proves that this polynomial is nonnegative on $\K$. After fixing the maximal degree of the SOS polynomials, computing their coefficients boils down to solving a semidefinite program (SDP). }
\new{An SDP problem involves a linear objective function with constraints over symmetric matrices with nonnegative eigenvalues. It can be solved with interior-point methods, yielding polynomial time algorithms at prescribed accuracy. 
For more details about applications of SDP together with complexity estimates, we refer to~\cite{NN-94,Vandenberghe94SDP,deKlerkSDP}. }
\new{A well-known limitation of the Lasserre 's hierarchy is due to the size of the SDP matrices involved in the SOS decompositions. For a system involving polynomials with $n$ variables of maximal degree $k$, this size grows rapidly as it is proportional to $\binom{n+k}{n}$. To overcome these limitations, several research efforts have been pursued to take into account the properties of certain classes of structured systems, e.g. sparsity~\cite{Waki06SparseSOS,Las06SparseSOS} or symmetry~\cite{Riener2013SymmetricSDP}. In particular, previous work by the author~\cite{toms16} exploits the special structure of the roundoff error function by applying the sparse variant of the first Lasserre's hierarchy to the linear part.}

While the first SDP hierarchy allows to approximate from above the maximum of a polynomial,~\cite{Lasserre11} provides a second complementary SDP hierarchy, yielding a sequence of converging lower bounds. 
\new{At each step of the second hierarchy, the lower bound on the maximum of a given polynomial is computed by solving a so-called {\em generalized eigenvalue problem}. Given two symmetric matrices $\A$ and $\B$ with known entries, this consists of finding the smallest value of $\lambda$ such that the matrix $\lambda \A - \B$ has only nonnegative eigenvalues. In our context, the two matrices encode certain information regarding the moments of some probability measure $\mu$ supported on the set of constraints $\K$. For instance, in the bivariate case, the entries of these matrices at the second step of the hierarchy necessarily depend on the value of the integrals $\int_\K y_1 d \mu$, $\int_\K y_2 d \mu$, $\int_\K y_1^2 d \mu$, $\int_\K y_1 y_2 d \mu$ and $\int_\K y_2^2 d \mu$. In several cases, the value of these integrals are available analytically. This includes the case where $\mu$ is the uniform (also called Lebesgue) measure and $\K$ is the unit box $[0, 1]^n$ (or any product of real closed intervals), the simplex or the euclidean ball. The interested reader can find more details about these closed formula in~\cite{Grundmann78,Parsimony16,deKlerk16}. }
\new{By contrast with the first Lasserre's hierarchy, the second one cannot easily handle the case where $\K$ is defined by a general set of polynomial inequality constraints. For instance, computing the moments of the uniform measure on a polytope is NP-hard (see e.g.~\cite{DeLoera10}). This is still an open problem to design a hierarchy yielding certified lower bounds in the general case. To the best of our knowledge, there is also no variant of the second Lasserre's hierarchy exploiting the properties of special problems.}

\new{Several efforts have been made to provide convergence rates  for the two hierarchies. For the first hierarchy, the theoretical estimates from~\cite{Nie07Putinar} yield convergence rates of $O(1/\sqrt[c]{\log (2 k/c)})$, where $c$ is a constant depending only on $\K$ (but not explicitely known) and $k$ is the selected step of the hierarchy. This yields very pessimistic bounds by contrast with the results obtained for practical case studies. The situation is rather different for the complexity analysis~\cite{deKlerk16} of the second Lasserre's hierarchy as the convergence rates are no worse than $O(1/\sqrt{r})$ and often match practical experiments. Investigating the gap between the two hierarchies, either from a theoretical or practical point of view, could provide insights on how to improve the estimates of the first Lasserre's hierarchy.}

\new{
Following this line of research, the motivation of this paper relates both to the roundoff error analysis and the use of SDP relaxations dedicated to sparse polynomial problems. On the one hand, we focus on deriving a sparse variant of the second Lasserre's hierarchy for the special case of roundoff error computation. On the other hand, we aim at providing insights regarding the gap between this variant and the hierarchy of upper bounds from~\cite{toms16}. 
}
\paragraph{\textbf{Contributions}}
We provide an SDP hierarchy inspired from~\cite{Lasserre11} to obtain a sequence of converging lower bounds on the \new{largest} absolute roundoff error obtained with a particular rounding model. This hierarchy and the one developed in~\cite{toms16} complement each other as  the combination of both now allows to enclose the \new{largest} absolute roundoff error in smaller and smaller intervals. 

We release a software package called $\fpsdp$\footnote{\url{https://github.com/magronv/FPSDP}} implementing this SDP hierarchy.

The rest of the article is organized as follows: in Section~\ref{sec:background} we provide preliminary background about floating-point arithmetic and SDP, allowing to state the considered problem of roundoff error. This problem is then addressed in Section~\ref{sec:robustsdp} with our SDP hierarchy of converging lower bounds. Section~\ref{sec:benchs} is devoted to numerical experiments in order to compare the performance of our $\fpsdp$ software with existing tools.

\section{FLOATING-POINT ARITHMETIC AND SEMIDEFINITE PROGRAMMING}
\label{sec:background}
\subsection{Floating-Point Arithmetic and Problem Statement}
\label{sec:fp_pb}
Let us denote by $\varepsilon$ the \new{machine epsilon or unit roundoff}, $\R$ the field of real numbers and $\F$ the set of binary floating-point numbers. Both overflow and \new{subnormal} range values are neglected.
Under this assumption, any real number $x \in \R$ is approximated with its closest floating-point representation $\hat{x} = x (1 + e)$, with $|e| \leq \varepsilon$ and $\hat{\cdot}$ being the rounding operator. 
\new{This can be selected among either rounding toward zero, rounding toward $\pm\infty$ or rounding to nearest. From now on, for the sake of simplicity we only consider rounding to nearest}.  
We refer to~\cite{higham2002accuracy} for related background. 

The number $\varepsilon := 2^{-\prec}$ bounds from above the relative floating-point error, with $\prec$ being called the {\em precision}. For single (resp.~double) precision floating-point, the value of the \new{unit roundoff} is $\varepsilon = 2^{-24}$ (resp.~$\varepsilon = 2^{-53}$).

To comply with IEEE 754 standard arithmetic~\cite{IEEE}, for each real-valued operation $\bop_\R \in \{+, -, \times, \slash \}$, the result of the corresponding floating-point operation $\bop_\F \in \{\oplus, \ominus, \otimes, \oslash \}$ satisfies:
\begin{equation}
\label{eq:roundbop}
\bop_\F \, (\hat{x}, \hat{y}) = \bop_\R \, (\hat{x}, \hat{y}) \, (1 + e) \enspace, \quad |e| \leq \varepsilon = 2^{-\prec} \enspace.
\end{equation}

\paragraph{\textbf{Semantics}} Our program semantics is based on the encoding of  polynomial expressions in the~$\realtofloat$ software~\cite{toms16}. The input variables of the program are constrained within interval floating-point bounds.

We denote by \texttt{C} the type for numerical constants, being chosen between double precision floating-point and arbitrary-size rational numbers. This type \texttt{C} is used for the interval bounds and for the polynomial coefficients.

As in~\cite[Section 2.1]{toms16}, the type \texttt{pexprC} of polynomial expressions is the following inductive type:
\begin{lstlisting}
type pexprC = Pc of C | Px of positive | $-$ pexprC 
| $\,$pexprC$\,-\,$pexprC  | pexprC$\,+\,$pexprC | pexprC$\, \times \,$pexprC
\end{lstlisting}
The constructor \texttt{Px} allows to represent any input variable $x_i$ with the positive integer $i$.

\paragraph{\textbf{Interval enclosures for bounds of roundoff errors}}
Let us consider a program implementing a polynomial function $f(\x)$ of type $\texttt{pexprC}$ (with the above semantics), which depends on input variables $\x := (x_1, \dots, x_n)$ constrained in a box, i.e.~a product of closed \new{(real)} intervals $\X := [\underline{x_1}, \overline{x_1}] \times \dots [\underline{x_n}, \overline{x_n}]$.
After rounding each coefficient and elementary operation involved in $f$, we obtain a polynomial rounded expression denoted by $\hat{f}(\x, \e)$, which depends on the input variables $\x$ as well as additional roundoff error variables $\e := (e_1, \dots, e_m)$. Following~\eqref{eq:roundbop}, each variable $e_i$ belongs to the interval $[-\varepsilon, \varepsilon]$, thus $\e$ belongs to $\E := [-\varepsilon, \varepsilon]^m$.

Here, we are interested in bounding from below the absolute roundoff error $|r(\x, \e)| := | \hat{f}(\x, \e) - f (\x) |$ over  all possible input variables $\x \in \X$ and roundoff error variables $\e \in \E$. 
Let us define $\K := \X \times \E$ and let $r^\star$ stands for the maximum of $|r(\x, \e)|$ over $\K$, that is $r^\star := \max_{(\x, \e) \in \K} | r(\x, \e)|$. 

Note that when $\e = 0$ and $\x$ corresponds to floating-point input, the value of $| \hat{f}(\x, \e) - f (\x) |$ is \new{zero}, yielding the valid lower bound $0 \leq r^\star$. For instance, let us consider $f(x) = x/2$, for $x \in [1, 2]$. When $x$ is a
floating-point number, this function has no roundoff error. 
However, higher lower bounds can be obtained when such conditions are not fulfilled, e.g.~for non floating-point input values. 

Following the same idea used in~\cite{fptaylor15,toms16}, we first decompose the error term $r$ as the sum of a term $l(\x,\e)$, which is linear w.r.t.~$\e$, and a nonlinear term $h(\x,\e) := r(\x,\e) - l(\x,\e)$. 
Then a valid lower bound on $r^\star$ can be derived by using the reverse triangular inequality:
\begin{align}
\begin{split}
\label{eq:lhoptim} 
r^\star \geq \max_{(\x, \e) \in \K} |l(\x, \e)| - \max_{(\x, \e) \in \K} |h(\x, \e)| =:  l^\star - h^\star  \enspace.
\end{split}
\end{align}
We emphasize the fact that $h^\star$ is \new{in general} negligible compared to $l^\star$ since $h$ contains products of error terms with degree at least 2 (such as $e_i e_j$), thus can be bounded by $O(\epsilon^2)$. This bound is likely much smaller than the roundoff error induced by the linear term $l$.
To compute a bound on $h^\star$, it is enough in practice to compute second-order derivatives of $r$ w.r.t.~$\e$ then use Taylor-Lagrange inequality to get an interval enclosure of $h$ as in~\cite{fptaylor15} or~\cite{toms16}. Doing so, one obtains an upper bound on $h^\star$. \new{In practice, we compute $h^\star$ with the implementation of $\iabound$ available in the $\realtofloat$ software package.}

Then, subtracting this upper bound to any lower bound on $l^\star$ yields a valid lower bound on $r^\star$ for the particular multiplicative rounding model which is considered here. 
Hence, from now on, we focus on approximating the bound $l^\star$ of the linear term.
\begin{figure}[!t]
\begin{algorithmic}[1]                    
\Require input variables $\x$, input box $\X$, polynomial $f$, rounded polynomial $\hat{f}$, error variables $\e$, error box $\E$, relaxation procedure $\lowerbound$, relaxation order $k$
\Ensure lower bound on the \new{largest} absolute roundoff error $\mid \hat{f} - f \mid$ over $\K := \X \times \E$
\State Define the absolute error $r(\x, \e) := \hat{f}(\x,\e) - f(\x)$ \label{line:r}
\State Compute $l(\x,\e) := \sum_{j=1}^m \frac{\partial r(\x,\e)} {\partial e_j} (\x,0) \, e_j$ and $h := r - l$ \label{line:l}
\State Compute an upper bound on $h^\star$: $\overline{h} := \iaboundfun{h}{\K}$ \label{line:iabound}
\State Compute a lower bound on $\overline{l}$ : $\overline{l}_k := \lowerboundfun{l}{\K}{k}$  \label{line:lowerbound}
\State Compute an upper bound on $\underline{l}$: $\underline{l}_k := - \lowerboundfun{-l}{\K}{k}$  \label{line:upperbound}
\State Compute a lower bound on $l^\star$ : $l_k := \new{\max \{ -\underline{l}_k,  \overline{l}_k\}} $ 
\State \Return $\new{\max\{l_k-  \overline{h},0 \}} $ \label{line:fpsdpbound}
\end{algorithmic}
\caption{\texttt{fpsdp}: our algorithm to compute lower bounds of absolute roundoff errors for polynomial programs.}
\label{alg:fpsdp}
\end{figure}
The framework~\cite{toms16} allows to obtain a hierarchy of converging upper bounds of $l^\star$ using SDP relaxations. By contrast with~\cite{toms16}, our goal is to compute a hierarchy of converging lower bounds on $l^\star$.
For the sake of clarity, we define $\underline{l} := \min_{(\x, \e) \in \K} l(\x, \e)$ and $\overline{l} := \max_{(\x, \e) \in \K} l(\x, \e)$.
Computing $l^\star$ can then be cast as follows:
\begin{align}
\begin{split}
\label{eq:loptim}
l^\star := \max_{(\x, \e) \in \K} | l(\x, \e) | =  \max \{
|\underline{l}| , 
|\overline{l}|
\} \,.
\end{split}
\end{align}
Note that the computation of $\underline{l}$ can be formulated as a maximization problem since $\underline{l} := \min_{(\x, \e) \in \K} l(\x, \e) = - \max_{(\x, \e) \in \K} - l(\x, \e) $. 
Thus, any method providing lower bounds on $\overline{l}$ can also provide upper bounds on $\underline{l}$, eventually yielding lower bounds on $l^\star$. 

We now present our main \texttt{fpsdp} algorithm, given in Figure~\ref{alg:fpsdp}. This procedure is similar to the algorithm implemented in the upper bound tool $\realtofloat$~\cite{toms16}, except that we obtain lower bounds on absolute roundoff errors. 
Given a program implementing a polynomial $f$ with input variables $\x$ being constrained in the box $\X$, the \texttt{fpsdp} algorithm takes as input $\x$, $\X$, $f$, the rounded expression $\hat{f}$ of $f$, the error variables $\e$ as well as the set $\E$ of bound constraints over $\e$.
The roundoff error $r := \hat{f} - f$ (Line\lineref{line:r}) is decomposed as the sum of a polynomial $l$ which is linear w.r.t.~the error variables $\e$ and a remainder $h$. 
As in~\cite{toms16,fptaylor15}, we obtain $l$ by computing the partial derivatives of $r$ w.r.t.~$\e$ (Line\lineref{line:l}).  The computation of the upper bound on $h^\star$ (Line\lineref{line:iabound}) is performed as explained earlier on, with the so-called procedure $\iabound$ relying on basic interval arithmetic. 
Our algorithm also takes as input a $\lowerbound$ procedure,  which computes lower bounds of the maximum of polynomials.
In our case, we use $\lowerbound$ in Line\lineref{line:lowerbound} (resp.~Line\lineref{line:upperbound}) to compute a lower (resp.~upper) bound on $\overline{l}$ (resp.~$\underline{l}$). 
\new{Since zero is a valid lower bound for the largest absolute roundoff error, we return in Line\lineref{line:fpsdpbound} the maximal value between zero and the bound provided by \texttt{sdp\_bound}. This ensures that the \texttt{fpsdp} algorithm cannot return wrong results even if \texttt{sdp\_bound} returns bad error estimates.
}

In the sequel, we describe three possible instances of $\lowerbound$: \new{the first one relies on a hierarchy of generalized eigenvalue problems, the second one provides a hierarchy of bounds using only elementary computations and the third one is based on a hierarchy of semidefinite programming (SDP) relaxations. These three methods are described respectively} in Section~\ref{sec:geneig}, Section~\ref{sec:mvbeta} and Section~\ref{sec:robustsdp}. Each step of these  hierarchies is indexed by an integer $k$, called {\em relaxation order} and given as input to \texttt{fpsdp}.
\subsection{Existing Hierarchies of Lower Bounds for Polynomial Maximization}
\label{sec:lowerpop}
Here, we recall mandatory background explaining how to obtain  hierarchies of lower bounds for a given polynomial maximization problem~\cite{Lasserre11}.
Given $p \in \R[\y]$ a multivariate polynomial in $N$ variables $y_1, \dots, y_N$ and a box $\K := [\underline{y_1}, \overline{y_1}] \times \cdots \times [\underline{y_N}, \overline{y_N}]$, one considers the following polynomial maximization problem:
\begin{equation}
\label{eq:maxpop}
p^*  :=  \max_{\y \in \K} p (\y)  \,.
\end{equation}
The set of box constraints $\K \subseteq \R^N$ is encoded by
\[
\K := \{ \y \in \R^{N} : g_1 (\y) \geq 0, \dots, g_{N} (\y) \geq 0\} \enspace,
\]
for  polynomials $g_1 := (y_1 - \underline{y_1})(\overline{y_1} - y_1), \dots, g_N:= (y_N - \underline{y_N})(\overline{y_N} - y_N)$. 

For a given vector of $N$ nonnegative integers $\alpha \in \N^N$, we use the notation $\y^\alpha := y_1^{\alpha_1} \cdots y_N^{\alpha_N}$ and $|\alpha| := \sum_{i=1}^N \alpha_i$. Any polynomial $p \in \R[\y]$ of \new{total degree at most} $k$ can then be written as $p(\y) = \sum_{|\alpha| \leq k} p_{\alpha} \y^{\alpha}$. We write $\N_k^N := \{ \alpha \in \N^N : |\alpha| \leq k \}$. The cardinal of this set is equal to \new{$\binom{N+k}{k}= \frac{(N+k)!}{N! \, k!}$}.

We recall that a finite Borel measure $\mu$ on $\R^N$ is a nonnegative set function such that $\mu (\emptyset) = 0$, $\mu(\R^N)$ is finite, and $\mu$ is countably sub-additive. The support of $\mu$ is the smallest closed set $\K \subseteq \R^N$ such that $\mu (\R^N \backslash \K ) = 0$ (see~\cite{Royden88} for more details).

Let $\mu$ be a given finite Borel measure supported on $\K$ and $\z$ be the sequence of moments of $\mu$, given by $z_{\alpha} := \int_{\K} \y^\alpha d \mu(\y)$ for all $\alpha \in \N^N$. 
In some cases, one can explicitly compute $z_{\alpha}$ for each $\alpha \in \N^{N}$.
This includes the case when $\mu$ is the uniform measure with density 1, i.e.~$d \mu(\y) = d \y$, as $\K$ is a product of closed intervals.
For instance with $N = 2$, $\K = [0, 1]^2$ and $\alpha = (1, 0)$, one has $ z_{1,0} = \int_{\K}  y_1 \, d \y = \frac{1}{2}$.
With $\alpha = (2,1)$, one has $z_{2,1} = \int_{\K}  y_1^2 \, y_2 \, d \y = \frac{1}{3} \times \frac{1}{2} = \frac{1}{6}$. 

Given a real sequence $\z =(z_{\alpha})$, we define the multivariate linear functional $L_\z : \R[\y] \to \R$ by $L_\z(p) := \sum_{\alpha} p_{\alpha} z_{\alpha}$, for all $p \in \R[\y]$. 
For instance if $p(\y) := y_1^2 y_2 + 3 y_1 - \frac{2}{3} $, $\K = [0, 1]^2$ then $L_\z(p) = z_{2, 1} + 3 z_{1, 0} - \frac{2}{3} z_{0, 0} = \frac{1}{6} + \frac{3}{2} - \frac{2}{3} = 1$.
\paragraph*{Moment matrix}
The {\it moment} matrix $\M_k(\z)$ is the real symmetric matrix with rows and columns indexed by $\N_k^N$ associated with a sequence
$\z =(z_{\alpha})$, whose entries are defined by: 
\[ 
\M_k(\z)(\beta, \gamma) := L_{\z}(\y^{\beta + \gamma})  \,, \quad
\forall \beta, \gamma \in \N_k^N \,.  
\]
\new{
We rely on the graded lexicographic order to compare the elements of $\N_k^N$. That is, we use the order which first compares the total degree (sum of all entries), and in case of a tie apply lexicographic order. For $N= k = 2$, this gives $(0,0) < (1,0) < (0,1) < (2,0) < (1,1) < (0,2)$.
}

\paragraph*{Localizing matrix}
The {\it localizing} matrix associated with a sequence
$\z = (z_{\alpha})$ and a polynomial $p \in\R[\y]$ (with $p(\y)=\sum_{\alpha} p_{\alpha} \y^\alpha$)
is the real symmetric matrix $\M_k(p \, \z)$ with rows and columns indexed by $\N_k^N$, and whose entries are defined by: 
\[ 
\M_k(p \, \z) (\beta,\gamma) := \new{L_\z (p(\y) \, \y^{\beta + \gamma})}, \quad
\forall \beta, \gamma \in \N_k^N \,.
\]
The size of $\M_k(p \, \z)$ is equal to the cardinal of $\N_k^N$, i.e.~$\binom{N+k}{k}$. Note that when $p = 1$, one retrieves the moment matrix as \new{a} special case of localizing matrix. 
\begin{example}
\label{ex:matloc}
With $p(\y) := y_1^2 y_2 + 3 y_1 - \frac{2}{3} $, $\K = [0, 1]^2$ and $k=1$, one has 
$\M_1(\z) = \begin{psmallmatrix}
1 & \frac{1}{2} & \frac{1}{2} \\
\frac{1}{2} & \frac{1}{3} & \frac{1}{4} \\
\frac{1}{2} & \frac{1}{4} & \frac{1}{3}
\end{psmallmatrix} $
and $\new{\M_1(p \, \z)} = \begin{psmallmatrix}
1 & \frac{19}{24} & \frac{19}{36} \\
\frac{19}{24} & \frac{113}{180} & \frac{5}{12} \\
\frac{19}{36} & \frac{5}{12} & \frac{13}{36}
\end{psmallmatrix} $.
\new{Here the elements of $\N_1^2$ indexing the rows and columns of both matrices are $(0,0)$, $(1,0)$ and $(0,1)$, corresponding to the monomials $1$, $y_1$ and $y_2$, respectively.
}
For instance, the bottom-right corner of the localizing matrix $\M_1(\z)$ is obtained by computing $L_\z ( p(\y) \, y_2^2 ) = z_{2,3} + 3 z_{1,2} - \frac{2}{3} z_{0,2} = \frac{1}{12} + \frac{1}{2} - \frac{2}{3}\times\frac{1}{3} = \frac{13}{36}$.
\end{example}
Next, we briefly recall two existing methods to compute lower bounds of $p^\star$ as defined in~\eqref{eq:maxpop}.
\subsubsection{Hierarchies of generalized eigenvalue problems}
\label{sec:geneig}
Let us denote by $\R^{n \times n}$ the vector space of $n \times n$ real matrices.
For a symmetric matrix $\M \in \Mc_{n}(\R)$, the notation $\M \succeq 0$ \new{means that} $\M$ is semidefinite positive (SDP), i.e.~has only nonnegative eigenvalues. The notation $\A \succeq \B$ stands for $\A - \B \succeq 0$.
A semidefinite optimization problem is an optimization problem where the cost is a linear function and the constraints state that some given matrices are semidefinite positive (see~\cite{Vandenberghe94SDP} for more details about SDP).

The following sequence of SDP programs can be derived from~\cite{Lasserre11}, for each $k \in \N$:
\begin{equation}
\label{eq:supgeneigpop}
\begin{aligned}
\lambda_k (p) := \min\limits_{\lambda} \quad & \lambda \\			
\text{s.t.} 
\quad &  \lambda \, \M_k (\z) \succeq \M_k (p \, \z) \,, \\
\quad & \lambda \in \R \,.
\end{aligned}
\end{equation}
The only variable of Problem~\eqref{eq:supgeneigpop} is $\lambda$ together with a single SDP constraint of size $\binom{N+k}{N}$. 
This constraint can be rewritten as $ \, \M_k ( (\lambda - p) \z) \succeq 0$ by linearity of the localizing matrices.
Solving Problem~\eqref{eq:supgeneigpop} allows to obtain a non-decreasing sequence of lower bounds which converges to the global \new{maximum} $p^\star$ of the polynomial $p$. Problem~\eqref{eq:supgeneigpop}  is a {\em generalized eigenvalue} problem. As mentioned in~\cite[Section 2.3]{mvbeta16}, the 
computation of the number $\lambda_k(p)$ requires at most $O \bigl( \binom{N + k}{k}^3  \bigr) $ floating-point operations (flops).
\begin{theorem}{(~\cite[Theorem 4.1]{Lasserre11})}
\label{th:supgeneigpop}
For each $k \in \N$, Problem~\eqref{eq:supgeneigpop} admits an optimal solution $\lambda_k(p)$. Furthermore, the sequence $(\lambda_k(p))$ is monotone non-decreasing and $\lambda_k(p) \uparrow p^\star$ as $k \to \infty$.
\end{theorem}
The convergence rate have been studied later on  in~\cite{deKlerk16}, which states that $p^\star - \lambda_k(p) = O(\frac{1}{\sqrt{k}})$.
\begin{example}
\label{ex:geneig}
With $p(\y) := y_1^2 y_2 + 3 y_1 - \frac{2}{3} $, $\K = [0, 1]^2$,  we obtain  the following sequence of lower bounds: $\lambda_1 (p) = 0.82 \leq \lambda_2 (p) = 1.43 \leq \lambda_3 (p) = 1.83 \leq \dots \leq \lambda_{20}(p) = 2.72 \leq p^\star = \frac{10}{3}$.  The computation takes $16.2s$ on an Intel(R) Core(TM) i5-4590 CPU @ 3.30GHz.
Here, we notice that the convergence to the maximal value $p^\star$ is slow in practice, confirming what the theory suggests.
\end{example}
\if{
\new{
\paragraph{Verification of positive definitess}
to compute the solution of Problem~\eqref{eq:supgeneigpop}, we have to solve a generalized eigenvalue problem. For the sake of efficiency, this is practically done with numerical procedures such as~\texttt{eig} available in {\sc Matlab}. When such a procedure returns an (approximate) solution $\tilde{\lambda}$, one can add a specific constant (very possibly small) verify afterwards that the matrix $\lambda$, we rely on the following   ~\cite{Rump2006}
}
}\fi
\subsubsection{Hierarchies of bounds using elementary computations}
\label{sec:mvbeta}
By contrast with the above method, further work by~\cite{mvbeta16} provides a second method only requiring elementary computations. This method also yields a monotone non-decreasing sequence of lower bounds converging to the global maximum of a polynomial $p$ while considering  for each $k \in \N$:
\begin{equation}
\label{eq:supmvbetapop}
p_k^H :=  \min_{(\eta, \beta) \in \N_{2 k}^{2 N}} \sum_{|\alpha| \leq d} p_{\alpha} \frac{\gamma_{\eta+\alpha, \beta}}{\gamma_{\eta, \beta}}
\,,
\end{equation}
where, for each $(\eta, \beta) \in \N_{2 k}^{2 N}$ the scalar $\gamma_{\eta, \beta}$ is the corresponding moment of the measure whose density is the multivariate beta distribution:
\begin{equation}
\label{eq:mvbeta}
\gamma_{\eta, \beta} :=  \int_\K \y^{\eta}  \, (\bone - \y)^{\beta} d \y = \int_\K y_1^{\eta_1} \cdots  y_N^{\eta_N} \, (1 - y_1)^{\beta_1} \cdots (1 - y_N)^{\beta_N} d \y \,.
\end{equation}
As mentioned in~\cite[Section 2.3]{mvbeta16}, the computation of the number $p_k^H$ requires at most $O\bigl( \binom{2N + 2 k - 1}{2 k} \bigr)$ floating-point operations (flops).
\begin{theorem}{(~\cite[Lemma 2.4,Theorem 3.1]{mvbeta16})}
\label{th:supmvbetapop}
The sequence $(p_k^H)$ is monotone non-decreasing and $p_k^H \uparrow p^\star$ as $k \to \infty$.
\end{theorem}
As for the sequence $(\lambda_k(p))$, the convergence rate is also in $O(\frac{1}{\sqrt{k}})$ (see~\cite[Theorem~4.9]{mvbeta16}).
\begin{example}
\label{ex:mvbeta}
With $p(\y) := y_1^2 y_2 + 3 y_1 - \frac{2}{3} $, $\K = [0, 1]^2$,  we obtain the  following sequence of 20 lower bounds: $p_1^H = 0.52 \leq p_2^H = 0.95 \leq p_3^H = 1.25 \leq \dots \leq p_{20}^H = 2.42 \leq p^\star = \frac{10}{3}$.  The computation takes $17.1s$ on \new{the same machine as Example~\ref{ex:geneig}.}
For small order values \new{($k \leq 10$)}, this method happens to be more efficient than the one  previously used in~Example~\ref{ex:geneig} but yields coarser bounds.
At higher order (\new{$k \geq 10$}), both methods happen to yield similar accuracy and performance with a slow rate of convergence. 
Note that performance could be improved in both cases by vectorizing our implementation code.
\end{example}
\section{A SPARSE SDP HIERARCHY FOR LOWER BOUNDS OF ROUNDOFF ERRORS}
\label{sec:robustsdp}
\if{
$f$ over a simple  semialgbraic set (e.g.~box) $\mathbf{K} = \mathbf{X} \times \mathbf{E}$, when $f$ has linear dependency on the variables in $\mathbf{E}$. By contrast with the converging sequence of bounds in~[J.B. Lasserre, A new look at nonnegativity on closed sets and polynomial optimization, SIAM J. Optim. 21, pp. 864--885, 2010], we prove that nonnegativity of $f$ over $\mathbf{K}$ is equivalent to semidefinite positiveness of countably many uncertain moment matrices, with perturbations in $\mathbf{E}$.
A method relying on robust SDP relaxations to provide lower bounds over the absolute roundoff error of programs implementing polynomial functions. We compare the accuracy and efficiency of these robust relaxations 
}\fi
This section is dedicated to our main theoretical contribution, that is a new SDP hierarchy of converging lower bounds on the absolute roundoff error of polynomial programs. \new{This hierarchy exploits the sparsity pattern occurring in the definition of $l$.}

The two existing SDP hierarchies presented in Section~\ref{sec:lowerpop} can be directly applied to solve Problem~\eqref{eq:loptim}, that is the computation of lower bounds on $l^\star := \max_{(\x, \e) \in \K} | l(\x, \e) |$. 
In our case, $N = n + m$ is the sum of  the number of input and error variables, $p = l$ and $\y = (\x, \e) \in \K = \X \times \E$.
At order $k$, a first relaxation procedure, denoted by \texttt{geneig}, returns the number $\lambda_k(l)$ by solving Problem~\eqref{eq:supgeneigpop}. A second relaxation procedure, denoted by \texttt{mvbeta}, returns the number $l_k^H$ by solving Problem~\eqref{eq:supmvbetapop}. In other words, this already gives two implementations \texttt{geneig} and \texttt{mvbeta} for the relaxation procedure $\lowerbound$ in the algorithm \texttt{fpsdp} presented in Figure~\ref{alg:fpsdp}.

However, these two procedures can be computationally  demanding to get precise bounds for programs with larger number of variables, i.e.~for either higher values of $k$ \new{($\simeq 10$)} or $N= n + m$ \new{($\simeq 50$)}. Experimental comparisons performed in Section~\ref{sec:benchs} will support this claim.
The design of a third implementation is motivated by the fact that both \texttt{geneig} and \texttt{mvbeta} do not take directly into account the special \new{sparse} structure of the polynomial $l$, that is the linearity w.r.t.~$\e$.

We first note that $l(\x, \e) = \sum_{j=1}^m e_j s_j(\x)$, for polynomials $s_1, \dots, s_m \in \R[\x]$. The maximization problem $\overline{l} := \max_{(\x, \e) \in \K} l(\x, \e)$ can then be written as follows:
\begin{equation*}
\begin{aligned}
\overline{l}  := \min_{\lambda} \quad & \lambda \\
\text{s.t.} 
\quad & \lambda \ge  \sum_{j=1}^m e_j \, s_j(\x) \,, \quad \forall \x \in \X \,, \forall \e \in \E \,, \\
\quad & \lambda \in \R \,.
\end{aligned}
\end{equation*}
From now on, we denote by $(\z^{\X})$ the moment sequence associated with the uniform measure on $\X$. 
We first recall the following useful property of the localizing matrices associated to $\z^{\X}$:
\begin{property}
\label{th:loc}
Let $f \in \R[\x]$ be a polynomial. Then f is nonnegative over $\X$ if and only if $\M_k (f \, \new{\z^{\X}}) \succeq 0$, for all $k \in \N$.
\end{property}
\begin{proof}
This is a special case of~\cite[Theorem 3.2 (a)]{Lasserre11} applied to the uniform measure supported on $\X$ with moment sequence \new{$\z^{\X}$}.
\end{proof}
In particular for $f = 1$, Property~\ref{th:loc} states that the moment matrix $\M_k (\new{\z^{\X}})$ is semidefinite positive, for all $k \in \N$.
Let us now consider the following hierarchy of optimization programs, indexed by $k \in \N$:
\begin{equation}
\label{eq:robsupgeneigpop}
\begin{aligned}
\lambda_k' (l) := \min_{\lambda} \quad & \lambda \\
\text{s.t.} 
\quad & \lambda \, \M_k(\z^{\X}) \succeq \sum_{j=1}^{m} e_j \,  \M_k(s_j \, \z^{\X}) \,, \quad \forall \e \in \E \,,\\
\quad & \lambda \in \R \,.
\end{aligned}
\end{equation}
Problem~\eqref{eq:robsupgeneigpop} is called \new{{\em a robust SDP} program} as it consists of minimizing the (worst-case) cost while satisfying SDP constraints for each possible value of the parameters $\e$ within the box $\E$.
\begin{lemma}
\label{th:rob1}
For each $k \in \N$, Problem~\eqref{eq:robsupgeneigpop} admits a finite optimal solution $\lambda_k'(l)$. Furthermore, the sequence $(\lambda_k'(l))$ is monotone non-decreasing and $\lambda_k'(l) \uparrow \overline{l}$ as $k \to \infty$.
\end{lemma}
\begin{proof}
The proof is inspired from the one of~\cite[Theorem 4.1]{Lasserre11} since Problem~\eqref{eq:robsupgeneigpop} is a robust variant of Problem~\eqref{eq:supgeneigpop}. 

First, let us define for all $\e \in \E$ the polynomial $l_{\e} (\x) := l(\x, \e)$ in $\R[\x]$. The polynomial $\overline{l} - l$ is nonnegative over $\X \times \E$, thus for all $\e \in \E$, the polynomial $\overline{l} - l_{\e}$ is nonnegative over $\X$.
By using Property~\ref{th:loc}, all localizing matrices of $\overline{l} - l_{\e}$ are semidefinite positive. 
This yields $\M_k( (\overline{l} - l_{\e})  \z^{\X}) \succeq 0$, for all $\e \in \E$. By linearity of the localizing matrix, we get $\overline{l} \,  \M_k(\z^{\X}) \succeq \sum_{j=1}^{m} e_j \,  \M_k(s_j \, \z^{\X}) $, for all $\e \in \E$. For all $k \in \N$, this proves that $\overline{l}$ is feasible for Problem~\eqref{eq:robsupgeneigpop} and  $\lambda_k'(l) \leq \overline{l}$. Next, let us fix $k \in \N$ and an arbitrary feasible solution $\lambda$ for Problem~\eqref{eq:robsupgeneigpop}. Since for all $\e \in \E$, one has $\M_k( (\lambda - l_{\e})  \z^{\X}) \succeq 0$, this is in particular the case for $\e = 0$, which yields $\lambda \M_k(\z^{\X}) \succeq 0$. Since the moment matrix $\M_k(\z^{\X})$ is semidefinite positive, one has $\lambda \geq 0$. Thus the feasible set of Problem~\eqref{eq:robsupgeneigpop} is nonempty and bounded, which proves the existence of a finite optimal solution $\lambda_k'(l)$.

Next, let us fix $k \in \N$. For all $\e \in \E$, $\M_k( (\lambda - l_{\e})  \z^{\X})$ is a sub-matrix of $\M_{k+1}( (\lambda - l_{\e})  \z^{\X})$, thus $\M_{k+1}( (\lambda - l_{\e})  \z^{\X})  \succeq 0$ implies that  $\M_{k}( (\lambda - l_{\e})  \z^{\X})  \succeq 0$, yielding $\lambda_k'(l) \leq \lambda_{k+1}'(l)$. 
Hence, the sequence $(\lambda_k'(l))$ is monotone non-decreasing. 
Since for all $k \in \N$, $\lambda_k'(l) \leq \overline{l}$, one has $(\lambda_k'(l))$ converges to $\lambda'(l) \leq \overline{l}$ as $k \to \infty$. 
For all $\e \in \E$, for all $k \in \N$, one has $\M_k( (\lambda'(l) - l_{\e})  \z^{\X})  \succeq \M_k( (\lambda_k'(l) - l_{\e})  \z^{\X})  \succeq 0$. 
By using again Property~\ref{th:loc}, this shows that for all $\e \in \E$, the polynomial $\lambda'(l) - l_{\e}$ is nonnegative over $\X$, yielding $\lambda'(l) \geq \overline{l}$ and the desired result $\lambda'(l) = \overline{l}$.
\end{proof}
\if{
Solving the above generalized eigenvalue problems can be done with efficient solvers but is still computationally demanding for large polynomial degree and number of variables. Here, we propose a robust version for the relaxations seen in Section~\eqref{sec:geneig}.\\
}\fi
Next, we use the framework developed in~\cite{RobustElGhaoui} to prove that for all $k\in \N$, Problem~\eqref{eq:robsupgeneigpop} is equivalent to the following SDP involving the additional real variable $\tau$: 
\begin{equation}
\label{eq:suprobustsdp}
\begin{aligned}
\lambda_k'' (l) := \min_{\lambda, \tau} \quad & \lambda \\
\text{s.t.} 
\quad &  \begin{pmatrix}
\lambda \, \M_k(\z^{\X}) - \tau \, \L_k \, \L_k^{\T} & \Rb_k^{\T} \\ 
 \Rb_k  & \tau \I
\end{pmatrix} \succeq 0 \,, \\
\quad &  \lambda, \tau \in \R \,.
\end{aligned}
\end{equation}
Both matrices $\L_k = [\L_k^1 \cdots \L_k^m]$ and 
$\Rb_k = [\Rb_k^1 \dots \Rb_k^m]^\T$ 
are obtained by performing a full rank factorization of the localizing matrix 
$\M_k(s_j \, \z^{\X})$ for each $j = 1, \dots, m$. This can be done e.g.~with the PLDL$^\T$P$^\T$ decomposition~\cite[Section 4.2.9]{Golub96} and is equivalent to \new{finding} two matrices $\L_k^j$ and $\Rb_k^j$ such that $\M_k (s_j \, \z^{\X}) = 2 \, \L_k^j \, \Rb_k^j$. For the sake of clarity we use the notations $\L_k$ and $\Rb_k$ while omitting the dependency of both matrices w.r.t.~$\z^{\X}$.

For each $j = 1, \dots, m$, the matrix $\L_k^j$ (resp.~${\Rb_k^j}$) has the same number of lines (resp.~columns) as $\M_k(s_j \, \z^{\X})$, i.e.~$\binom{n+k}{k}$, and the same number of columns (resp.~lines) as the rank $r_j$ of $\M_k (s_j \, \z^{\X} )$. 
The size of the \new{identity matrix} $\I$ is $m \binom{n+k}{k}$.
%
\begin{theorem}
\label{th:rob2}
For each $k \in \N$, Problem~\eqref{eq:suprobustsdp} admits a finite optimal solution $\lambda_k''(l) = \lambda_k'(l)$\new{, where $\lambda_k'(l)$ is the solution of Problem~\eqref{eq:robsupgeneigpop}.
}
Furthermore, the sequence $\lambda_k''(l)$ is monotone non-decreasing and $\lambda_k''(l) \uparrow \overline{l}$ as $k \to \infty$.
\end{theorem}
\begin{proof}
It is enough to prove the equivalence between Problem~\eqref{eq:suprobustsdp} and Problem~\eqref{eq:robsupgeneigpop} since then the result follows directly from Lemma~\ref{th:rob1}.
Let us note $\bzero := (0, \dots, 0) \in \R^n$.
Problem~\eqref{eq:robsupgeneigpop} can be cast as Problem~(4) in~\cite{RobustElGhaoui} with $x = (\lambda, \bzero)$, $F(x) =  \lambda \, \M_k(\z^{\X})$, $\Delta = \diag(0,\e)$, $c = (1,\bzero)$, $\mathcal{D} = \R^{(m+1) \times (m+1)}$, $\rho = 1$. 

In addition, the robust SDP constraint of Problem\eqref{eq:robsupgeneigpop} can be rewritten as \new{the linear matrix inequality~(8) in Section~3.1 of~\cite{RobustElGhaoui}}, i.e.~$F + L \, \Delta \, (I - D \, \Delta)^{-1} \, R \, + \, R^\T \, \Delta \, (I - D \, \Delta)^{-\T} \, L^\T \succeq 0$, with $L = \L_k$, $R = \Rb_k$ and $D = 0$.
Then, the desired equivalence result follows from~\cite[Theorem~3.1]{RobustElGhaoui}.
\end{proof}
This procedure provides a third choice, called $\robustsdp$, for the relaxation procedure $\lowerbound$ in the algorithm \texttt{fpsdp} presented in Figure~\ref{alg:fpsdp}.
\paragraph{Computational considerations} 
As for the \texttt{geneig} procedure, one also obtains the convergence rate $\overline{l} - \lambda_k''(l) = O(\frac{1}{\sqrt{k}})$. However, the resolution cost of Problem~\eqref{eq:suprobustsdp} can be smaller.
Indeed, from~\cite[Section 4.2.9]{Golub96}, the cost of each full rank factorization is cubic in each localizing matrix size, yielding a total factorization cost of $O \bigl(m \, \binom{n+k}{k}^3 \bigr)$ flops.
From~\cite[Section 11.3]{IPM94} the  SDP solving cost is proportional to the cube of the matrix size, yielding $O \bigl(m^3 \, \binom{n+k}{k}^3 \bigr)$ flops for Problem~\eqref{eq:suprobustsdp}. 
Hence, the overall cost of the $\robustsdp$ procedure is bounded by $O \bigl(m^3 \, \binom{n+k}{k}^3 \bigr)$ flops. This is in contrast with the cost of $O \bigl( \binom{n+m + k}{k}^3  \bigr) $ flops for \texttt{geneig} as well as the cost of \new{$O \bigl( \binom{2 n + 2 m + 2 k - 1}{2 k} \bigr)$} flops for \texttt{mvbeta}. In the sequel, we compare  these expected costs for several values of $n$, $m$ and $k$.

\section{EXPERIMENTAL RESULTS AND DISCUSSION}
\label{sec:benchs}
Now, we present experimental results obtained by applying our algorithm~\texttt{fpsdp} (see Figure~\ref{alg:fpsdp}) with the three relaxation procedures $\geneig$, $\mvbeta$ and $\robustsdp$ to various examples coming from physics, biology, space control, and optimization. 
The procedures $\geneig$, $\mvbeta$ and $\robustsdp$ provide lower bounds of a polynomial maximum by solving Problem~\eqref{eq:supgeneigpop}, Problem~\eqref{eq:supmvbetapop} and Problem~\eqref{eq:suprobustsdp}, respectively.
The \texttt{fpsdp} algorithm is implemented as a software package written in $\matlab$, called $\fpsdp$. Setup and usage of $\fpsdp$ are described on the dedicated web-page\footnote{\url{https://github.com/magronv/FPSDP}} with specific instructions\footnote{see the \texttt{README.md} file in the top level directory}. 
The three procedures are implemented using $\yalmip$~\cite{yalmip}  which is a toolbox for advanced modeling and solution of \new{(non-)}convex optimization problems in $\matlab$. 

The only procedure which requires an SDP solver is $\robustsdp$. Indeed, $\geneig$ requires to solve a particular SDP instance which can be handled with a generalized eigenvalue solver.
In practice, $\geneig$ relies on the function $\texttt{eig}$ from $\matlab$ to solve generalized eigenvalue problems. 
Full rank factorization within the $\robustsdp$ procedure is performed with the function $\texttt{rref}$ from $\matlab$. 

For solving SDP problems, we rely on $\mosek \ 7.0$~\cite{mosek}. For more details about the installation, usage and setup instructions of \new{
$\yalmip$ and $\mosek$, we refer to the respective dedicated web-pages  \footnote{\url{https://docs.mosek.com/7.0/toolbox/index.html}}  \footnote{\url{http://users.isy.liu.se/johanl/yalmip/pmwiki.php?n=Main.WhatIsYALMIP}}.
} 

\if{ OLD
%
}\fi
\new{
Note that the solution of Problem~\eqref{eq:suprobustsdp} is computed with a numerical solver implemented with finite-precision.
Hence, we have to check the bounds obtained with $\mosek$ by verifying that the matrix involed in~\eqref{eq:suprobustsdp} has nonnegative eigenvalues. We compute the eigenvalues of the matrices provided by $\mosek$ with $\texttt{eig}$. If the minimal eigenvalue returned by $\texttt{eig}$ is negative, then we add it to the lower bound and round the result towards $-\infty$. We emphasize that this procedure has an insignificant impact on the value of lower bounds computed for all benchmarks.
To perform rigorous checking of the positive definitess of the matrices, one could rely e.g.~on  verified floating-point Cholesky's decomposition, as proposed in~\cite{Rump2006}, but this is beyond the scope of this paper.
}
\subsection{Benchmark Presentation}
All examples are displayed in \new{the appendix. Our} results have been obtained on an Intel(R) Core(TM) i5-4590 CPU @ 3.30GHz. For the sake of further presentation, we associate an alphabet character (from \texttt{a} to \texttt{i}) to identify each of the 9 polynomial nonlinear programs which implement polynomial functions: \texttt{a} and \texttt{b} come from physics, \texttt{c}, \texttt{d} and \texttt{e} are derived from expressions involved in the proof of Kepler Conjecture~\cite{Flyspeck06} and \texttt{f}, \texttt{g} and \texttt{h} implement polynomial approximations of the sine and square root functions. All programs are used for similar upper bound comparison in~\cite[Section 4.1]{toms16}. Each program implements a polynomial $f$ with  $n$ input variables $\x \in \X$ and yields after rounding $m$ error variables $\e \in \E = [-\varepsilon, \varepsilon]^m$.
\begin{example}
The program \texttt{c} (see~\new{the appendix}) implements the polynomial expression 
\begin{align*}
f(\x) := x_2 \times x_5 + x_3 \times x_6 - x_2 \times x_3  - x_5 \times x_6 \\
+ x_1 \times ( - x_1 +  x_2 +  x_3  - x_4 +  x_5 +  x_6) \,,
\end{align*} 
and the program input is the six-variable vector $\x =  (x_1, x_2, x_3, x_4, x_5, x_6)$. The set $\X$ of possible input values is a product of closed intervals: $\X = [4.00, 6.36]^6$. 
The polynomial $f$ is obtained by performing 15 basic operations (1 negation, 3 subtractions, 6 additions and 5 multiplications).
When executing this program with a set $\hat{\x}$ of floating-point numbers defined by $\hat{\x} =  (\hat{x}_1, \hat{x}_2, \hat{x}_3, \hat{x}_4, \hat{x}_5, \hat{x}_6) \in \X$, one obtains the floating-point result $\hat{f}$. The error variables are $e_1, \dots, e_{21} \in [-\varepsilon, \varepsilon]$ and $\E = [-\varepsilon, \varepsilon]^{21}$.
\end{example}
For the sake of conciseness, we only considered to compare the performance of $\fpsdp$ on programs implemented in double ($\varepsilon = 2^{-53}$) precision floating point. \new{
We use two pre-processing features embedded in the $\realtofloat$ software package: first, we rely on the parser of $\realtofloat$ to get the expression of the linear part $l$ of the roundoff error $| f(\x)-\hat{f}(\x,\e) | = | l(\x,\e) + h(\x,\e) |$. Second, we compute an interval enclosure $h^\star$ of $h$ with the sub-rountine $\iabound$ available in $\realtofloat$ (as recalled in Section~\ref{sec:fp_pb}).}
\if{
OLD
To compute lower bounds of the roundoff error $| f(\x)-\hat{f}(\x,\e) | = | l(\x,\e) + h(\x,\e) |$, 
we use~\texttt{Real2Float} to obtain the expression $l$ and to bound $h$. We refer to~\cite[Section 3.1]{toms16} for more details.
}\fi

\begin{table}[!t]
\begin{center}
\tbl{Comparison results of lower bounds and execution times (in seconds) for \new{largest} absolute roundoff errors among  $\geneig$, $\mvbeta$, $\robustsdp$ and \new{$\nlopt$}.
\label{table:error}}{{\small
\begin{tabular}{cc|cccccc|c|c|c}
\hline
\texttt{id} & k  & \multicolumn{2}{c}{$\geneig$} & \multicolumn{2}{c}{$\mvbeta$} & \multicolumn{2}{c|}{$\robustsdp$} & \multicolumn{1}{c|}{\new{$\nlopt$}} & \multicolumn{1}{c|}{lower} & \multicolumn{1}{c}{upper} \\
& & bound & time & bound & time & bound & time & \new{bound}  & bound  & bound \\
\hline
\multirow{5}{*}{\texttt{a}}
& 1 & $1.05\text{e--}15$ & 0.63 & $1.85\text{e--}16$ & 0.31 & $3.30\text{e--}14$ & 0.75 & \multirow{5}{*}{\new{$4.80\text{e--}13$}}  & \multirow{5}{*}{$2.28\text{e--}13$}  &  \multirow{5}{*}{$5.33\text{e--}13$} \\
& 2 & $2.84\text{e--}14$ & 1.69 &  $4.07\text{e--}15$ & 20.9 & $7.52\text{e--}14$ & 0.79 & & & \\
& 3  & $5.83\text{e--}14$ & 29.7 & $8.88\text{e--}15$ & 645. & $1.10\text{e--}13$ & 1.06 & & & \\
& 4 & $8.72\text{e--}14$ & >1e4 & $1.73\text{e--}14$ & >1e5 & $1.62\text{e--}13$ & 2.09 & & & \\
& 8 & $-$ & $-$ & $-$ & $-$ & ${3.55\text{e--}13}$ & 164. & & & \\
\hline
\multirow{5}{*}{\texttt{b}}
& 1  & $8.40\text{e--}14$ & 1.10 & $4.99\text{e--}14$ & 1.83 & $3.56\text{e--}12$ & 0.31  & \multirow{5}{*}{\new{$6.40\text{e--}11$}} & \multirow{5}{*}{$2.19\text{e--}11$} &  \multirow{5}{*}{$6.48\text{e--}11$} \\
& 2  & $1.31\text{e--}12$ & 2.75 &  $2.41\text{e--}13$ & 226. & $5.31\text{e--}12$ & 0.42 & & &  \\
& 3  & $2.89\text{e--}12$ & 288. & $5.08\text{e--}13$ & >1e5 & $8.04\text{e--}12$ & 1.04 & & & \\
& 4 &  $-$ & $-$ & $-$ & $-$ & $1.13\text{e--}11$ & 4.11 & & &  \\
& 7  & $-$ & $-$ & $-$ & $-$ & $2.60\text{e--}11$ & 152. & & & \\
\hline
\multirow{4}{*}{\texttt{c}}
& 1  & $9.45\text{e--}15$ & 1.25 & $3.95\text{e--}15$ & 4.34 & $9.68\text{e--}15$ & 0.73 & \multirow{4}{*}{\new{$1.02\text{e--}13$}} &  \multirow{4}{*}{$2.23\text{e--}14$} &  \multirow{4}{*}{$1.18\text{e--}13$}\\
& 2  & $1.64\text{e--}14$ & 37.3 & $7.89\text{e--}15$ & >1e4 & $1.48\text{e--}14$ & 6.69 & & & \\
& 3  & $-$ & $-$  & $-$  & $-$  & $2.62\text{e--}14$  & 172. & & & \\
& 4   & $-$ & $-$ & $-$ & $-$  & $-$  & $-$  & & & \\
\hline
\multirow{4}{*}{\texttt{d}}
& 1 & $3.01\text{e--}14$ & 1.63 & $1.41\text{e--}14$ & 13.6 & $1.49\text{e--}13$ & 1.67 & \multirow{4}{*}{\new{$3.93\text{e--}13$}}  & \multirow{4}{*}{$7.58\text{e--}14$} &  \multirow{4}{*}{$4.47\text{e--}13$}\\
& 2  & $5.38\text{e--}14$ & 163. & $2.45\text{e--}14$ & >4e4 & $2.22\text{e--}13$ & 3.73 & & & \\
& 3 & $-$ & $-$  & $-$  & $-$  & $3.04\text{e--}13$ & 33.3 &  & & \\
& 4  & $-$ & $-$ & $-$ & $-$  & ${4.06\text{e--}13}$  & 275. & & & \\
\hline
\multirow{4}{*}{\texttt{e}}
& 1  & $9.72\text{e--}28$ & 5.74 & $5.55\text{e--}14$ & 86.1 & $2.88\text{e--}13$ & 2.53 & \multirow{4}{*}{\new{$2.01\text{e--}12$}}  & \multirow{4}{*}{$3.03\text{e--}13$} &  \multirow{4}{*}{$2.09\text{e--}12$}\\
& 2 & $-$ & $-$  & $-$  & $-$  & $4.48\text{e--}13$  & 55.2 &  & & \\
& 3  & $-$ & $-$ & $-$ & $-$  & $-$  & $-$  &   & & \\
& 4 & $-$ & $-$  & $-$  & $-$  & $-$  & $-$  &   & & \\
\hline
\multirow{5}{*}{\texttt{f}}
& 1 & $8.34\text{e--}17$ & 0.89 & $1.50\text{e--}17$ & 0.59 & $1.98\text{e--}16$ & 1.28 & \multirow{5}{*}{\new{$5.50\text{e--}16$}} & \multirow{5}{*}{$4.45\text{e--}16$} &  \multirow{5}{*}{$6.03\text{e--}16$}\\
& 2 & $1.52\text{e--}16$ & 2.24 & $4.50\text{e--}17$ & 45.9 & $2.31\text{e--}16$ & 1.29 & & & \\
& 3 & $2.07\text{e--}16$ & 65.5 & $7.95\text{e--}17$ & >1e4 & $2.72\text{e--}16$ & 1.30 & & & \\
& 4 & $-$ & $-$ & $-$ & $-$ & $3.04\text{e--}16$ & 1.30 & & & \\
& 8 & $-$ & $-$ & $-$ & $-$ & ${4.43\text{e--}16}$ & 1.40 & & & \\
\hline
\multirow{6}{*}{\texttt{g}}
& 1 & $1.09\text{e--}16$ & 0.33 & $4.93\text{e--}17$ & 0.04 & $3.99\text{e--}16$ & 1.22 & \multirow{6}{*}{\new{$1.00\text{e--}15$}} & \multirow{6}{*}{$3.34\text{e--}16$} & \multirow{6}{*}{$1.19\text{e--}15$}\\
& 2 & $2.43\text{e--}16$ & 0.97 & $1.18\text{e--}16$ & 0.94 & $4.83\text{e--}16$ & 1.22 & & & \\
& 3 & $3.68\text{e--}16$ & 1.71 & $1.78\text{e--}16$ & 10.6 & $5.62\text{e--}16$ & 1.23 & & & \\
& 4 & $4.72\text{e--}16$ & 7.21 & $2.33\text{e--}16$ & 79.9 & $6.37\text{e--}16$ & 1.24 & & & \\
& 6 & $6.28\text{e--}16$ & 646. & $2.87\text{e--}16$ & >1e4 & $7.85\text{e--}16$ & 1.25 & & & \\
& 8 & $-$ & $-$ & $-$ & $-$ & ${9.30\text{e--}16}$ & 1.28 & & & \\
\hline
\multirow{5}{*}{\texttt{h}}
& 1 & $2.30\text{e--}16$ & 1.52 & $1.29\text{e--}16$ & 0.28 & $4.83\text{e--}16$ & 1.34 & \multirow{5}{*}{\new{$7.10\text{e--}16$}} & \multirow{5}{*}{$4.45\text{e--}16$} &  \multirow{5}{*}{$1.29\text{e--}15$}\\
& 2 & $4.00\text{e--}16$ & 2.87 & $2.55\text{e--}16$ & 28.4 & $5.40\text{e--}16$ & 1.35 & & & \\
& 3 & $5.36\text{e--}16$ & 161. & $3.28\text{e--}16$ & >1e4 & $5.78\text{e--}16$ & 1.39 & & & \\
& 4 & $-$ & $-$ & $-$ & $-$ & $6.13\text{e--}16$ & 1.40 & & & \\
& 7 & $-$ & $-$ & $-$ & $-$ & ${7.00\text{e--}16}$ & 1.50 & & & \\ 
\hline
\multirow{5}{*}{\texttt{i}}
& 1  & $3.07\text{e--}14$ & 0.58 & $1.60\text{e--}14$ & 0.61 & $1.56\text{e--}13$ & 1.32  & \multirow{5}{*}{\new{$1.42\text{e--}12$}}  & \multirow{5}{*}{$1.47\text{e--}13$} &  \multirow{5}{*}{$1.43\text{e--}12$} \\
& 2 & $6.71\text{e--}14$ & 2.48 & $2.68\text{e--}14$ & 42.5 & $2.13\text{e--}13$ & 1.39 & & & \\
& 3 & $1.25\text{e--}13$ & 47.1 & $3.72\text{e--}14$ & >1e4 & $2.84\text{e--}13$ & 1.43 & & & \\
& 4 & $1.92\text{e--}13$ & >1e4 & $5.35\text{e--}14$ & >2e5 & $3.63\text{e--}13$ & 1.62 & & & \\
& 8 & $-$ & $-$ & $-$ & $-$ & ${7.67\text{e--}13}$ & 7.34 & & &  \\
\hline
\end{tabular}
}
}
\end{center}
\end{table}
\begin{table}[!t]
\begin{center}
\tbl{Expected magnitudes of \new{flops counts} for the three procedures $\geneig$, $\mvbeta$ and $\robustsdp$.\label{table:flops}}{{\small
\begin{tabular}{lcccc|ccc}
\hline
{Benchmark} & id &  $n$ & $m$ &  $k$ & $\geneig$  & $\mvbeta$  & $\robustsdp$\\
\hline  
%
\multirow{2}{*}{\texttt{rigidBody1}} & \multirow{2}{*}{\texttt{a}} & \multirow{2}{*}{3}  & \multirow{2}{*}{10}  
& 1 & ${2.75\text{e+}03}$ & ${3.50\text{e+}03}$ & ${6.40\text{e+}04}$ \\
&  &  &  & 8 & ${8.43\text{e+}15}$ & ${1.04\text{e+}12}$ & ${4.50\text{e+}09}$\\
\hline
\multirow{2}{*}{\texttt{rigidBody2}} & \multirow{2}{*}{\texttt{b}} & \multirow{2}{*}{3}  & \multirow{2}{*}{15}  
& 1 & ${6.86\text{e+}03}$ & ${9.99\text{e+}03}$ & ${2.16\text{e+}05}$ \\
&  &  &  & 7 & ${1.12\text{e+}17}$ & ${1.02\text{e+}13}$ & ${5.84\text{e+}09}$\\
\hline
\multirow{2}{*}{\texttt{kepler0}} & \multirow{2}{*}{\texttt{c}} & \multirow{2}{*}{6}  & \multirow{2}{*}{21}  
& 1 & ${2.20\text{e+}04}$ & ${3.12\text{e+}04}$ & ${3.18\text{e+}06}$ \\
&  &  &  & 4 & ${3.12\text{e+}13}$ & ${6.19\text{e+}10}$ & ${8.58\text{e+}10}$\\
\hline
\multirow{2}{*}{\texttt{kepler1}} & \multirow{2}{*}{\texttt{d}} & \multirow{2}{*}{4}  & \multirow{2}{*}{28}  
& 1 & ${3.60\text{e+}04}$ & ${5.83\text{e+}04}$ & ${2.75\text{e+}06}$ \\
&  &  &  & 4 & ${2.05\text{e+}14}$ & ${2.98\text{e+}11}$ & ${7.53\text{e+}09}$\\
\hline
\multirow{2}{*}{\texttt{kepler2}} & \multirow{2}{*}{\texttt{e}} & \multirow{2}{*}{6}  & \multirow{2}{*}{42}  
& 1 & ${1.18\text{e+}05}$ & ${1.96\text{e+}05}$ & ${2.55\text{e+}07}$ \\
&  &  &  & 4 & ${1.99\text{e+}16}$ & ${9.99\text{e+}12}$ & ${6.87\text{e+}11}$\\
\hline
\multirow{2}{*}{\texttt{sineTaylor}} & \multirow{2}{*}{\texttt{f}} & \multirow{2}{*}{1}  & \multirow{2}{*}{13}  
& 1 & ${3.38\text{e+}03}$ & ${5.28\text{e+}03}$ & ${1.76\text{e+}04}$ \\
&  &  &  & 8 & ${3.27\text{e+}16}$ & ${3.45\text{e+}12}$ & ${1.61\text{e+}06}$\\
\hline
\multirow{2}{*}{\texttt{sineOrder3}} & \multirow{2}{*}{\texttt{g}} & \multirow{2}{*}{1}  & \multirow{2}{*}{6}  
& 1 & ${5.12\text{e+}02}$ & ${6.30\text{e+}02}$ & ${1.73\text{e+}03}$ \\
&  &  &  & 8 & ${2.67\text{e+}11}$ & ${4.08\text{e+}08}$ & ${1.58\text{e+}05}$\\
\hline
\multirow{2}{*}{\texttt{sqroot}} & \multirow{2}{*}{\texttt{h}} & \multirow{2}{*}{1}  & \multirow{2}{*}{15}  
& 1 & ${4.92\text{e+}03}$ & ${7.92\text{e+}03}$ & ${2.70\text{e+}04}$ \\
&  &  &  & 7 & ${1.48\text{e+}16}$ & ${2.51\text{e+}12}$ & ${1.73\text{e+}06}$\\
\hline
\multirow{2}{*}{\texttt{himmilbeau}} & \multirow{2}{*}{\texttt{i}} & \multirow{2}{*}{2}  & \multirow{2}{*}{11}  
& 1 & ${2.75\text{e+}03}$ & ${3.87\text{e+}03}$ & ${3.60\text{e+}04}$ \\
&  &  &  & 8 & ${8.43\text{e+}15}$ & ${1.14\text{e+}12}$ & ${1.22\text{e+}08}$\\
\hline
\end{tabular}
}}
\end{center}
\end{table}

Table~\ref{table:flops} compares expected magnitudes of \new{flop counts} for 
$\geneig$, $\mvbeta$ and $\robustsdp$, following from the study at the end of Section~\ref{sec:robustsdp}. 
For each program, we show the cost for the initial relaxation order $k = 1$ as well as for the highest one used for error computation in Table~\ref{table:error}. The results indicate that we can expect the procedure $\geneig$ to be more efficient at low relaxation orders while being outperformed by $\robustsdp$ at higher orders. Besides, the $\mvbeta$ procedure is likely to have performance lying in between the two others. We mention that the interested reader can find more detailed experimental comparisons between the two  relaxation procedures $\geneig$  and $\mvbeta$ in~\cite{mvbeta16}.
\subsection{Numerical Evaluation}
For each benchmark, Table~\ref{table:error} displays the quality of the roundoff error bounds with corresponding execution times. We emphasize that our $\fpsdp$ tool relies on the simple rounding model described in Section~\ref{sec:fp_pb}

Note that a head-to-head comparison between the three SDP relaxation procedures and $\sthreefp$~\cite{Chiang14s3fp} would be more difficult. 
Indeed, $\sthreefp$ relies on several possible heuristic search algorithms and  measures output errors after executing programs written in \texttt{C++} with certain input values.  
The rounding occurring while executing such programs is more likely to fit with an improved model, based for instance on a piecewise constant absolute error bound (see e.g.~\cite[Section 1.2]{toms16} for more explanation about such models). 

\new{We also compare the three procedures based on SDP with the competitive $\nlopt$~\cite{nlopt} software\footnote{\url{http://nlopt.readthedocs.io/en/latest}}. 
Since $l(\x,\e) = \sum_{j=1}^m s_j(\x) e_j$, the maximal absolute value of $l$ on $\X \times [-\varepsilon,\varepsilon]^m$ is equal to the maximum of the function $a(\x) := \varepsilon \sum_{j=1}^m |s_j(\x)|$ on $\X$. For the sake of efficiency, we execute  $\nlopt$ on $a$ with the optimization algorithm NLOPT\_GN\_DIRECT, which is derivative-free so that it can optimize functions with absolute values.
The results returned by $\nlopt$ are not necessarily upper nor lower bounds of $a$ but when the solver stops, it returns a point $\x^\star \in \X$. Then we obtain an interval enclosure of $a(\x^\star)$ in {\sc Matlab} which is a valid lower bound of $a$, thus a valid lower bound of $|l|$.
}
The lower bounds on the absolute error are obtained as in~\cite{Darulova14Popl}, by executing each program on several random inputs satisfying the input restrictions.
For comparison purpose, we also provide the best known upper bounds computed with SDP from~\cite[Table~II]{toms16}. 

\begin{figure}[!t]
\begin{minipage}[t]{0.49\textwidth}
\centering
\includegraphics[scale=\sizefig]{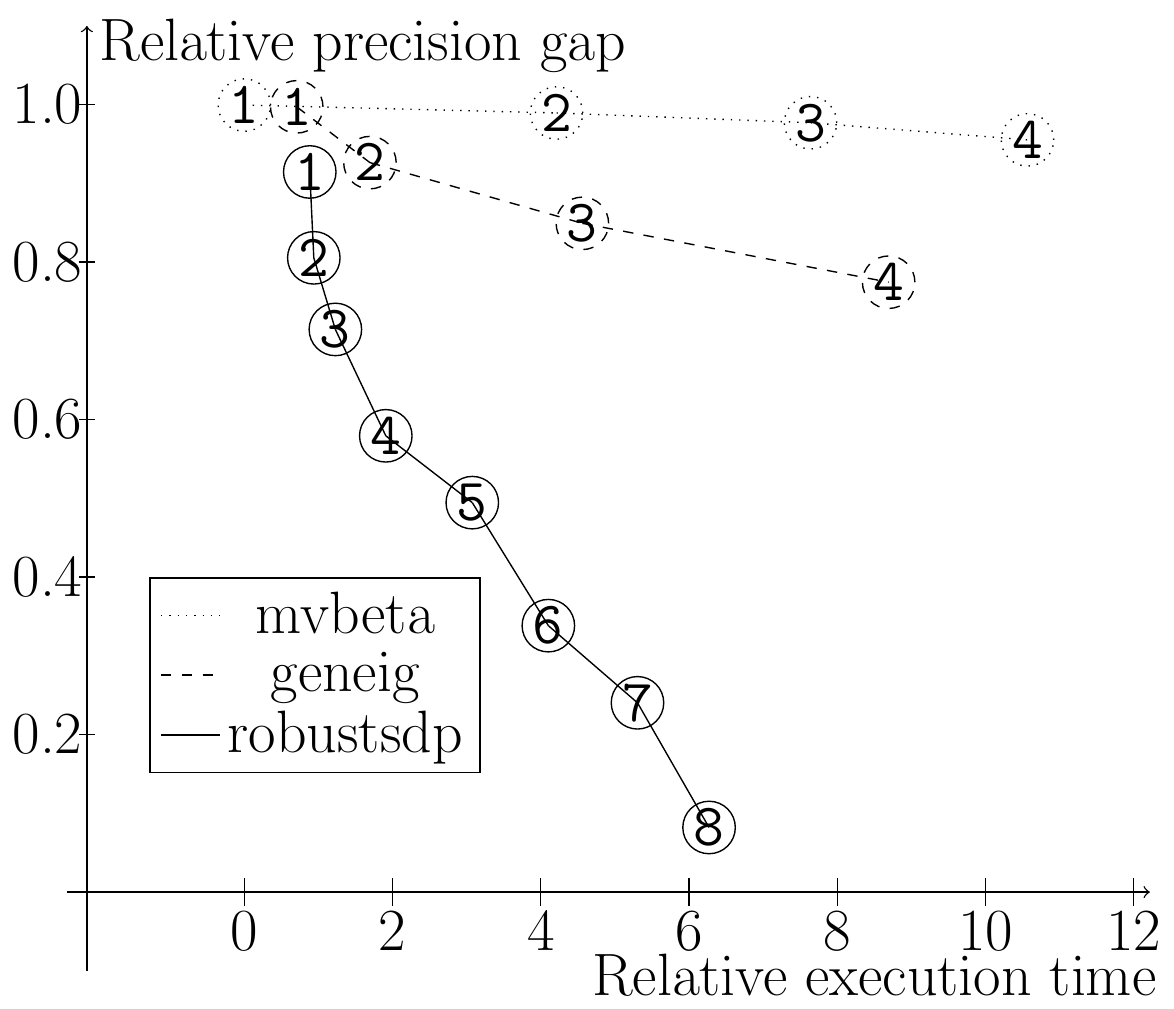}
\\(\texttt{a}) \label{fig:a}
\end{minipage}
\hspace*{\fill}
\begin{minipage}[t]{0.49\textwidth}
\centering
\includegraphics[scale=\sizefig]{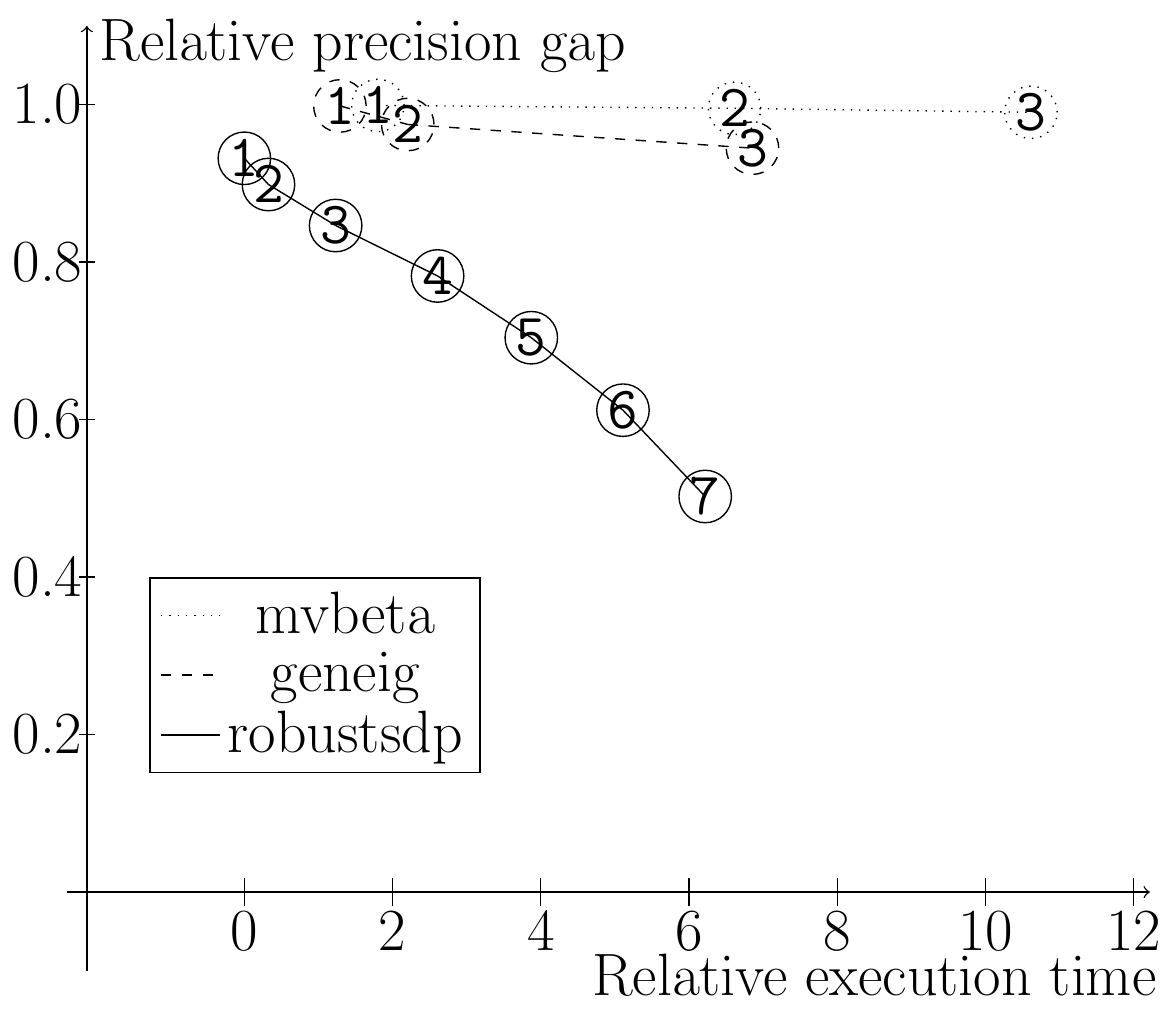}
\\{(\texttt{b})} \label{fig:b}
\end{minipage}
\begin{minipage}[t]{0.49\textwidth}
\centering
\includegraphics[scale=\sizefig]{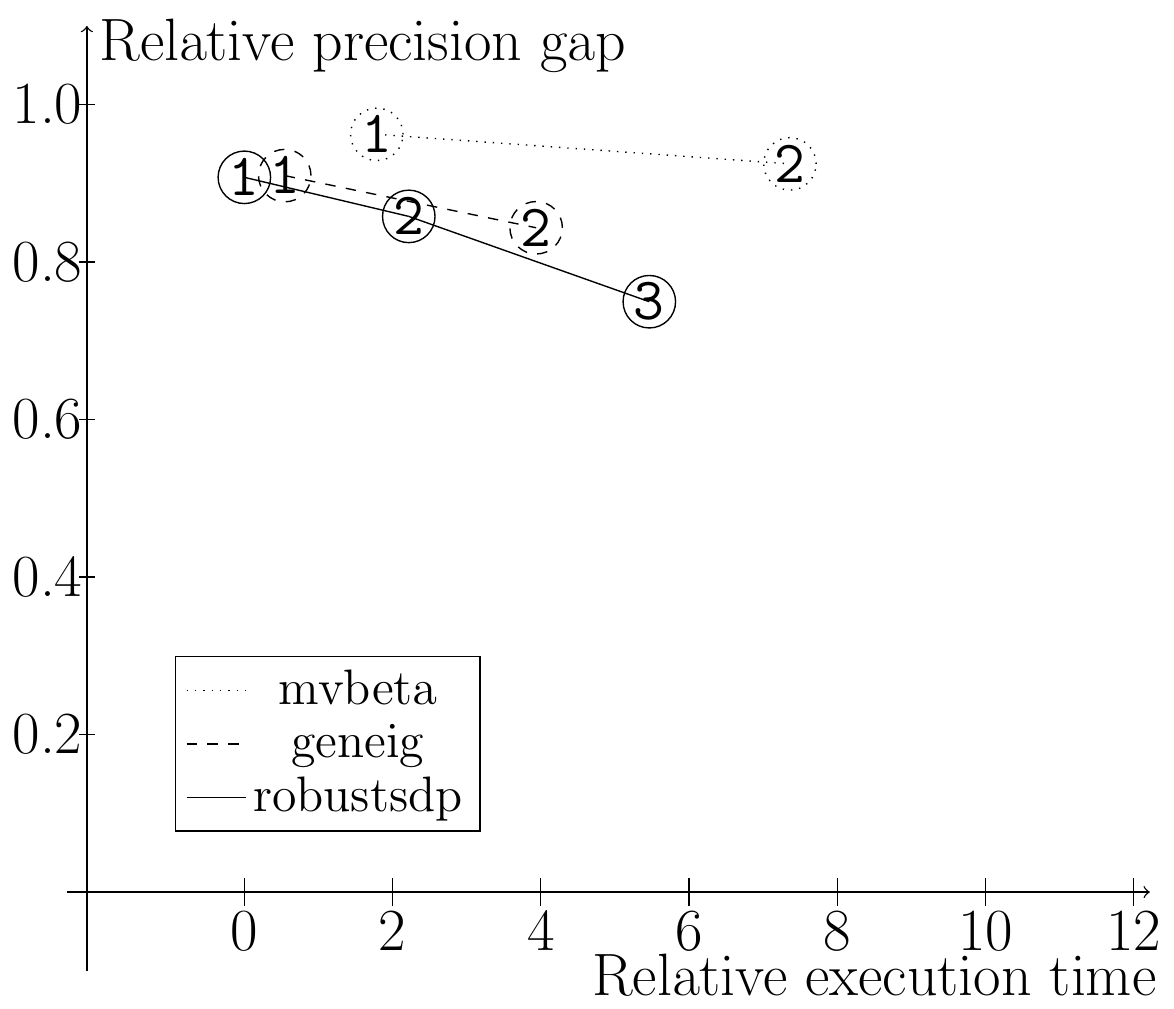}
\\{(\texttt{c})} \label{fig:c}
\end{minipage}
\hspace*{\fill}
\begin{minipage}[t]{0.49\textwidth}
\centering
\includegraphics[scale=\sizefig]{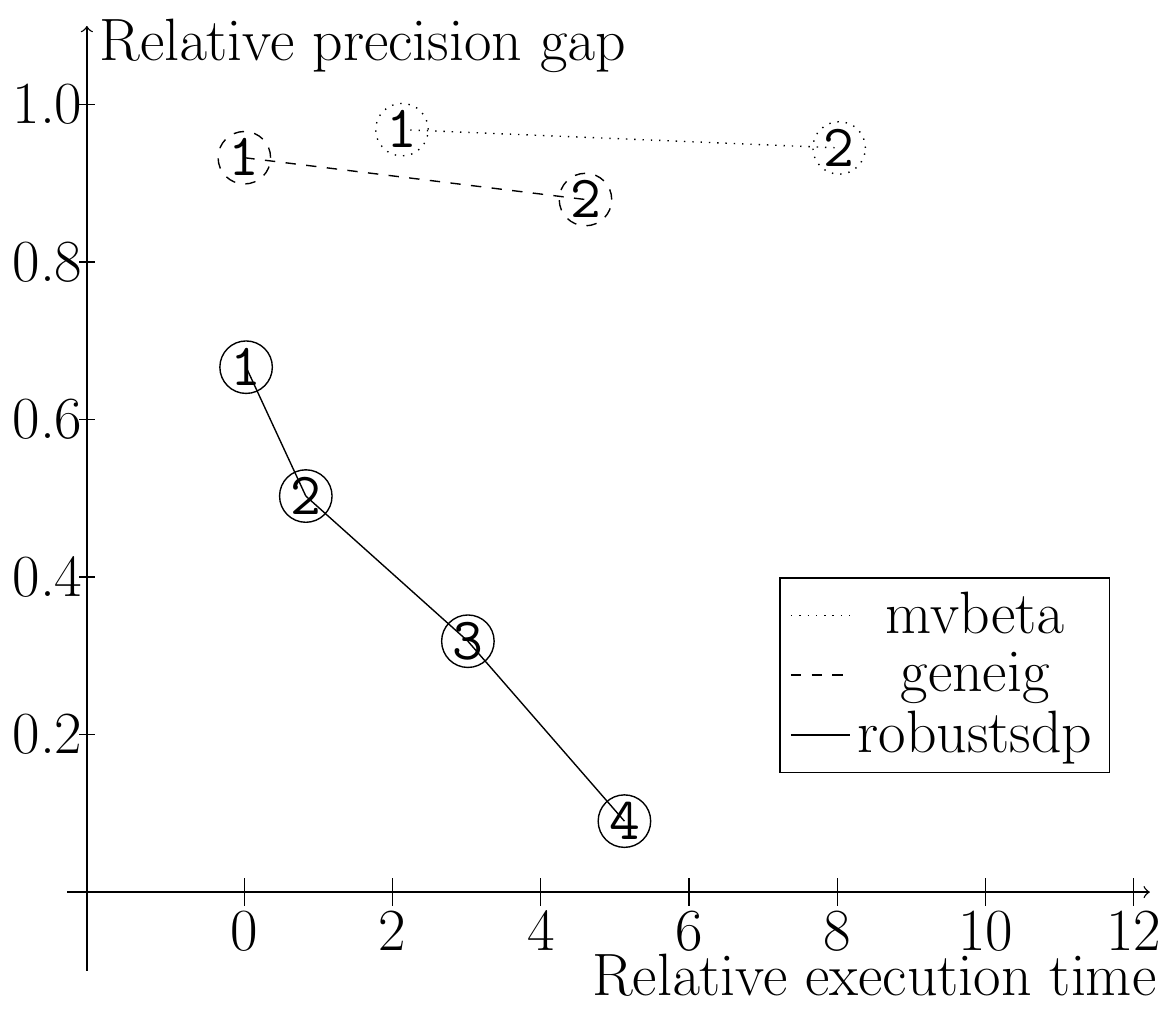}
\\{(\texttt{d})} \label{fig:d}
\end{minipage}
\begin{minipage}[t]{0.49\textwidth}
\centering
\includegraphics[scale=\sizefig]{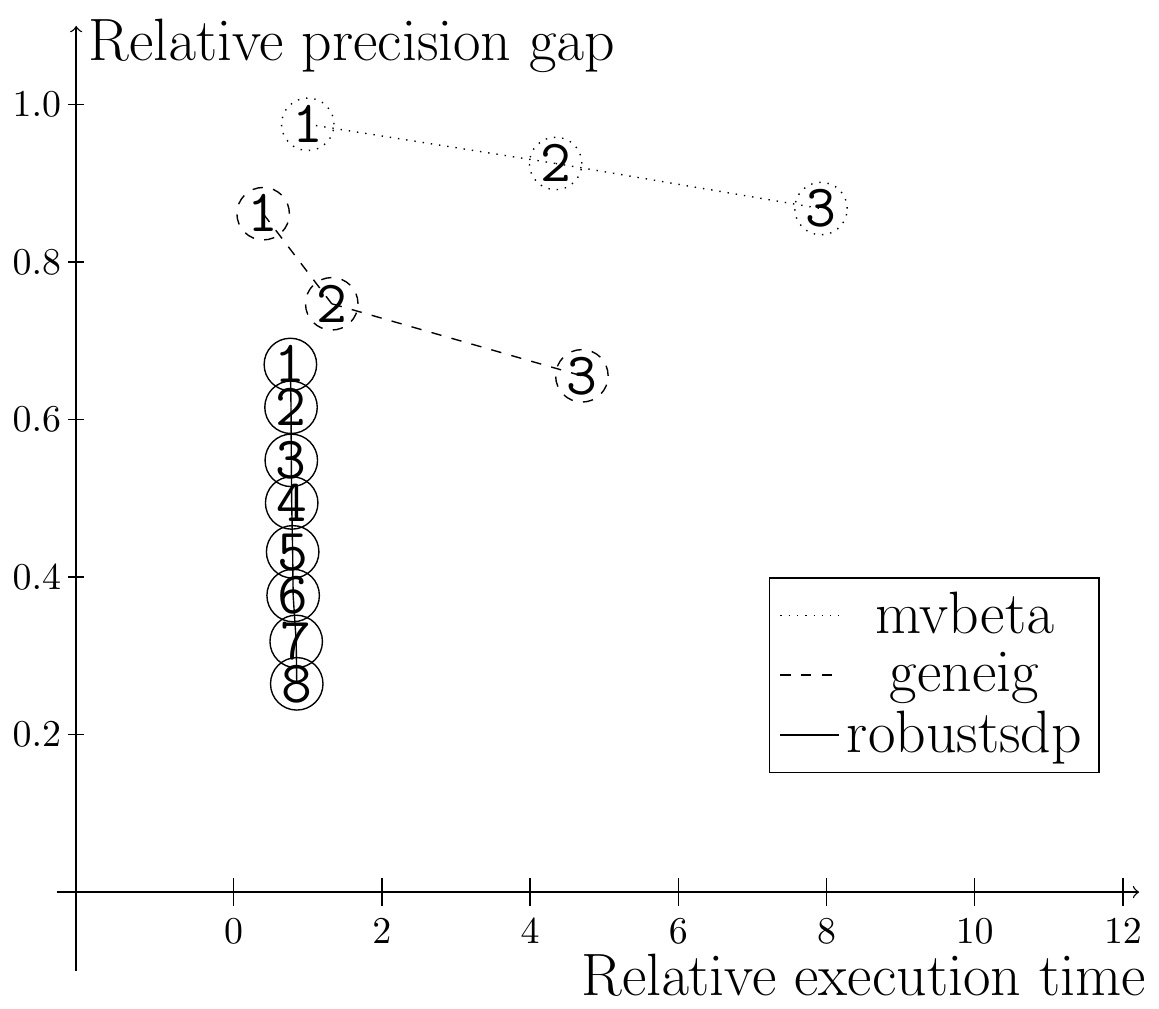}
\\{(\texttt{f})} \label{fig:f}
\end{minipage}
\hspace*{\fill}
\begin{minipage}[t]{0.49\textwidth}
\centering
\includegraphics[scale=\sizefig]{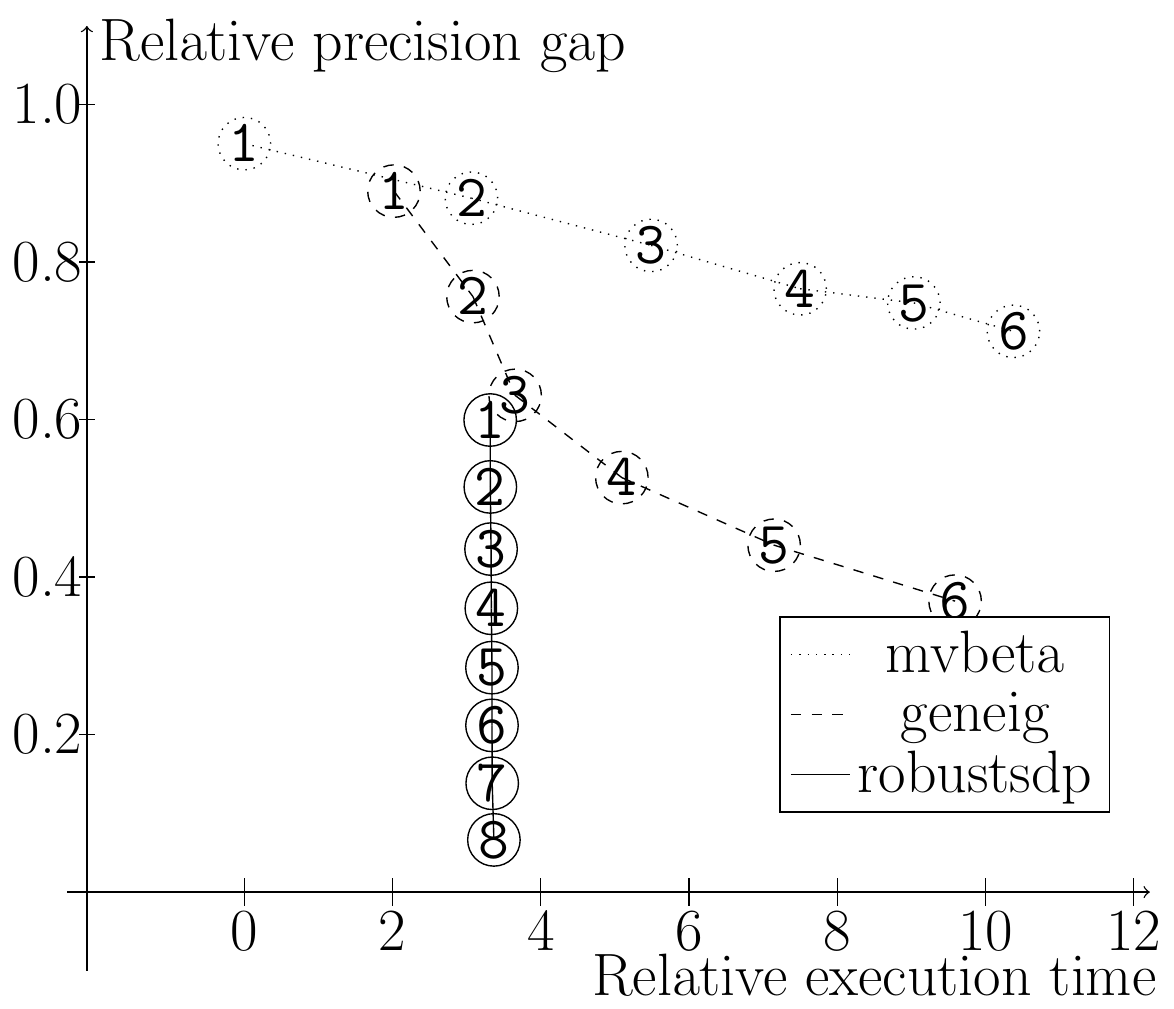}
\\{(\texttt{g})} \label{fig:g}
\end{minipage}
\begin{minipage}[t]{0.49\textwidth}
\centering
\includegraphics[scale=\sizefig]{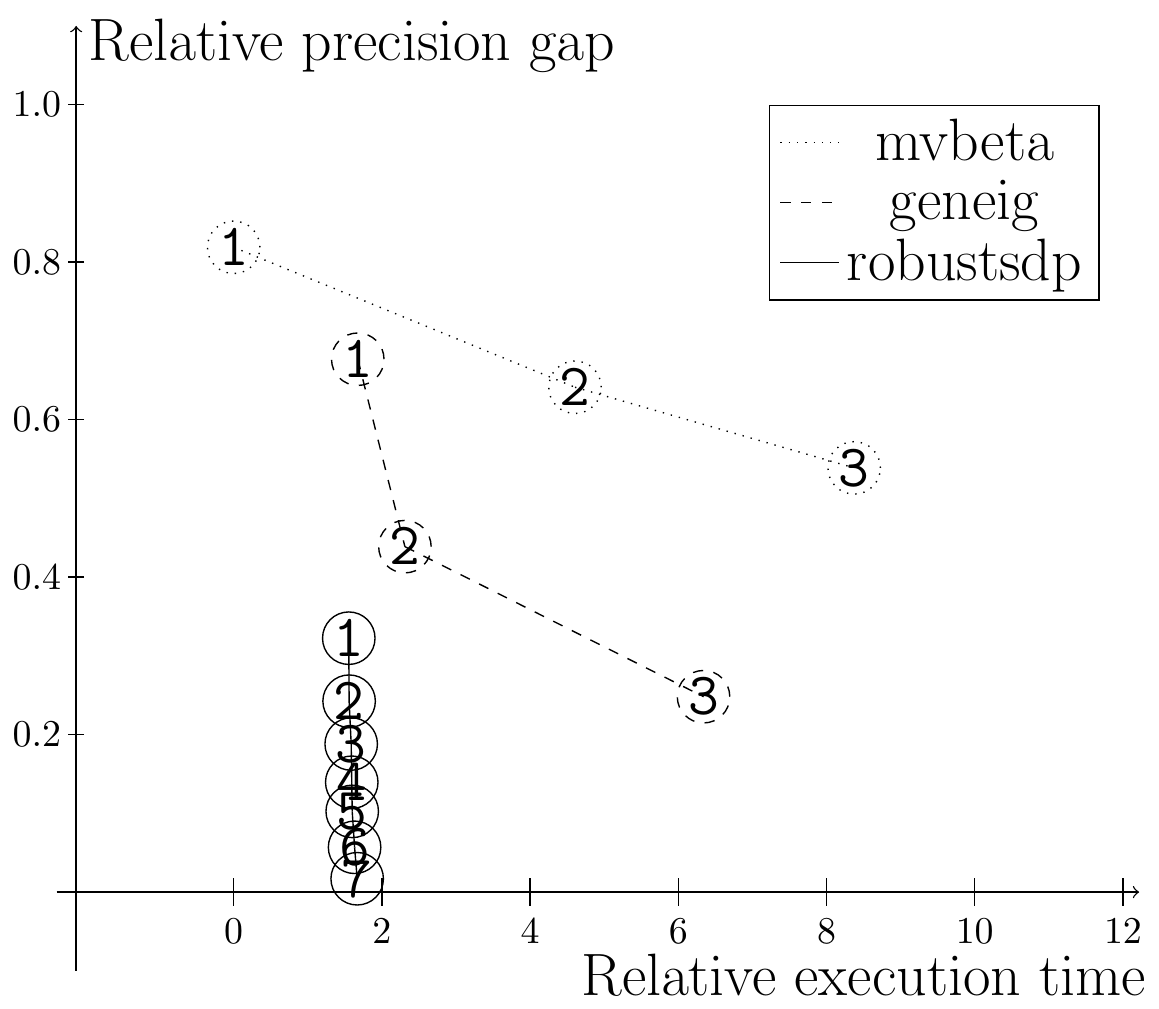}
\\{(\texttt{h})} \label{fig:h}
\end{minipage}
\hspace*{\fill}
\begin{minipage}[t]{0.49\textwidth}
\centering
\includegraphics[scale=\sizefig]{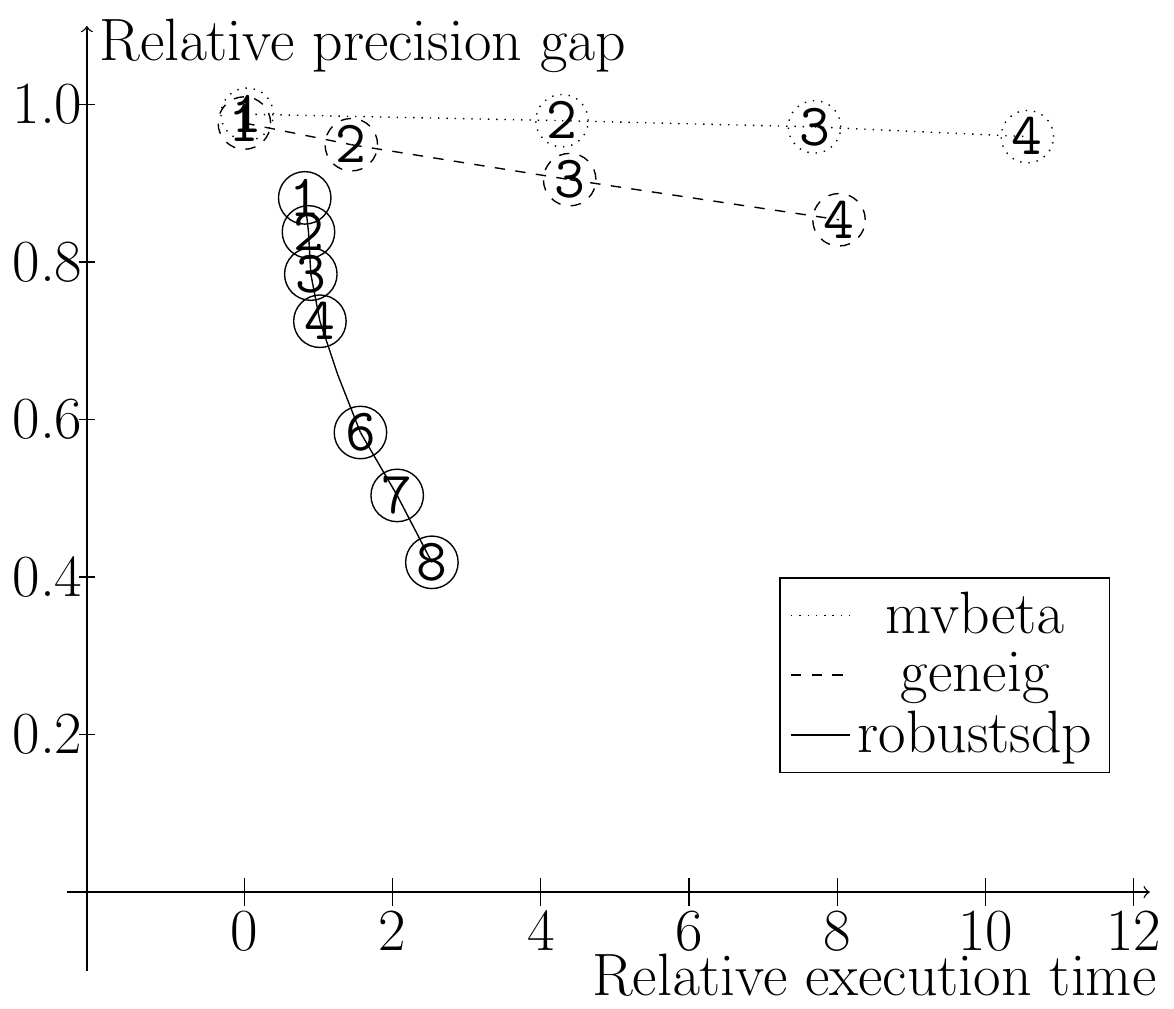}
\\{(\texttt{i})} \label{fig:i}
\end{minipage}
\caption{Relative gap and execution time results for benchmarks.}
\label{fig:graphs} 
\end{figure}

As shown in Table~\ref{table:error}, the $\robustsdp$ procedure is the most accurate among our three hierarchies and provides the tightest enclosure bounds for all programs. 
\new{The $\nlopt$ tool is always faster than all other methods. Our $\robustsdp$ procedure is more precise than $\nlopt$ for the benchmark \texttt{d}. Even though $\nlopt$ is the most accurate procedure for all other benchmarks, $\robustsdp$ often provides similar bounds. In particular, $\robustsdp$ is respectively 7 \% and 2 \% less precise than $\nlopt$ for the benchmarks \texttt{g} and \texttt{h}. 
}
%

For relaxations order greater than 2, 
$\robustsdp$  is faster \new{than the two other hierarchies} for all programs. Except for program \texttt{c}, either $\geneig$ or $\mvbeta$  yields better performance at the first relaxation order. 
%
%
The symbol ``$-$''  in a column entry means that we aborted the execution of the corresponding procedure after running more than $1\text{e}6$ seconds. 
Note that such \new{a} behavior systematically occurs when analyzing programs \texttt{f-i} with $\geneig$ and $\mvbeta$ at maximal relaxation orders. This confirms the expectation results from Table~\ref{table:flops}, as $\robustsdp$  yields more tractable SDP relaxations. 
Note that for these benchmarks, we performed experiments for each intermediate order $k$ between 4 and the maximal indicated one. For conciseness, we have not displayed all intermediate results in the table but use them later on in Figure~\ref{fig:graphs}. One way to increase the performance of the $\fpsdp$ tool would be to vectorize the current code which creates  moment/localizing matrices, instead of writing loop-based code.

The purpose of Figure~\ref{fig:graphs} is to emphasize the ability of $\fpsdp$ to make a compromise between accuracy and \new{performance}.
All program results (except \texttt{e} due to the lack of experimental data) are reported in Figure~\ref{fig:graphs}.  Each value of $k$ corresponds to a circled integer point. 
For each experiment, we define the three execution times $t_\geneig$, $t_\mvbeta$ and $t_\robustsdp$ and  the minimum $t$ among the three values. The x-axis coordinate of the circled point is $\ln\bigl(\frac{t_\geneig}{t}\bigr)$ for the $\geneig$ procedure (and similarly for the other procedures).
The corresponding lower bounds are denoted by $\varepsilon_\geneig$, $\varepsilon_\mvbeta$ and $\varepsilon_\robustsdp$. With $\overline{\varepsilon}$ being the reference upper bound, the y-axis coordinate of the circled point is the relative error gap for $\geneig$, i.e.~$r_\geneig := 1 - \frac{\varepsilon_\geneig}{\overline{\varepsilon}}$ and similarly for the other procedures.

For each $k$, the relative location of the corresponding circled integers indicate which procedure either performs better \new{(by being faster)} or provides more accurate bounds. 
%
In particular for \texttt{i}, $\robustsdp$ is less efficient than the two other procedures for first relaxation orders (relative execution time less than 1.5) then outperforms the other procedures. We also observe that $\mvbeta$ is more efficient at low relaxation orders for programs $\texttt{a}$ and $\texttt{g-i}$ as well as $\geneig$ for programs $\texttt{c-d}$, $\texttt{f}$  and $\texttt{i}$. The  procedure $\geneig$ is always more accurate than $\mvbeta$. 
%
%
%
\section{CONCLUSION AND PERSPECTIVES}
\new{We have presented three procedures, respectively based on hierarchies of generalized eigenvalue problems, elementary computations and semidefinite programming (SDP) relaxations}. These three methods allow to compute lower bounds of roundoff errors for programs implementing polynomials with input variables being box constrained. While the \new{first two} procedures are direct applications of existing methods in the context of polynomial optimization, the third one relies on a new hierarchy of robust SDP relaxations, allowing to tackle specifically the roundoff error problem.  Experimental results obtained with our $\fpsdp$ tool, implementing these three procedures, prove that SDP relaxations are able to provide accurate lower bounds in an efficient way. 

A first direction of further research is the extension of the SDP relaxation framework to programs \new{implementing non-polynomial functions}, with either finite or infinite loops. This requires to derive a hierarchy of inner converging SDP approximations for reachable sets of discrete-time polynomial systems in either finite or infinite horizon. Another topic of interest is the formal verification of lower bounds with a proof assistant such as $\coq$~\cite{CoqProofAssistant}. To achieve this goal, we could benefit from recent formal libraries~\cite{Denes2012} in computational algebra.
\if{
\section*{Acknowledgments}
This work has been partially supported by the LabEx PERSYVAL-Lab (ANR-11-LABX-0025-01) funded by the French program ``Investissement d'avenir'' and by the European Research Council (ERC) ``STATOR'' Grant Agreement nr. 306595.
}\fi


                                  
\appendix
\section*{APPENDIX: POLYNOMIAL PROGRAM BENCHMARKS}
\setcounter{section}{1}
\label{appendix}

\begin{itemize} 
	\item[\texttt{a}] rigibody1 : $(x_1,x_2,x_3) \mapsto -x_1x_2-2x_2x_3-x_1-x_3$ defined on $[-15,15]^3$.
	\item[\texttt{b}] rigibody2 : $(x_1,x_2,x_3) \mapsto 2x_1x_2x_3+6x_3^2-x_2^2x_1x_3-x_2$ defined on $[-15,15]^3$.
	\item[\texttt{c}] kepler0 : $(x_1,x_2,x_3,x_4,x_5,x_6) \mapsto x_2x_5+x_3x_6-x_2x_3-x_5x_6+x_1(-x_1+x_2+x_3-x_4+x_5+x_6)$ defined on $[4,6.36]^6$.
	\item[\texttt{d}] kepler1 : $(x_1,x_2,x_3,x_4) \mapsto x_1x_4(-x_1+x_2+x_3-x_4)+x_2(x_1-x_2+x_3+x_4)+x_3(x_1+x_2-x_3+x_4)-x_2x_3x_4-x_1x_3-x_1x_2-x_4$ defined on $[4,6.36]^4$.
	\item[\texttt{e}] kepler2 : $(x_1,x_2,x_3,x_4,x_5,x_6) \mapsto x_1x_4(-x_1+x_2+x_3-x_4+x_5+x_6)+x_2x_5(x_1-x_2+x_3+x_4-x_5+x_6)+x_3x_6(x_1+x_2-x_3+x_4+x_5-x_6)-x_2x_3x_4-x_1x_3x_5-x_1x_2x_6-x_4x_5x_6$ defined on $[4,6.36]^6$.
	\item[\texttt{f}] sineTaylor : $x \mapsto x-\frac{x^3}{6.0}+\frac{x^5}{120.0}-\frac{x^7}{5040.0}$ defined on $[-\frac{\tilde{\pi}}{2},\frac{\tilde{\pi}}{2}]$, with $\frac{\tilde{\pi}}{2} := 1.57079632679$.
	\item[\texttt{g}] sineOrder3 : $x \mapsto 0.954929658551372 x-0.12900613773279798 x^3$ defined on $[-2,2]$.
	\item[\texttt{h}] sqroot : $x \mapsto 1.0+0.5x-0.125x^2+0.0625x^3-0.0390625x^4$ defined on $[0,1]$.
	\item[\texttt{i}] himmilbeau : $(x_1,x_2) \mapsto (x_1^2+x_2-11)^2+(x_1+x_2^2-7)^2$ defined on $[-5,5]^2$.
	
\end{itemize} 
\bibliographystyle{ACM-Reference-Format-Journals}
\bibliography{lowerroundsdp}

\end{document}